\pdfoutput=1
\documentclass{amsart}
\usepackage{algorithm}
\usepackage{algpseudocode}
\usepackage{xr-hyper}
\usepackage{hyperref}

\usepackage{fancyvrb}
\usepackage{amsmath}
\usepackage{amsfonts}
\usepackage{amsthm}
\usepackage{amssymb}
\usepackage{array}
\usepackage{cases}
\usepackage{caption}
\usepackage{xcolor}
\usepackage{tikz}
\usepackage{tikz-cd}
\usepackage{mathtools}
\usetikzlibrary{calc}
\usepackage{centernot}
\usepackage{mathtools}
\usepackage{graphicx}
\usepackage[margin=1in]{geometry}
\usepackage[makeroom]{cancel}
\usepackage{stackengine}
\usepackage{mathrsfs}
\usepackage{enumitem}

\makeatletter
\newcommand{\sqcupplus@inner}[2]{%
  \vcenter{\hbox{\ooalign{%
    \hfil$\m@th#1\sqcup$\hfil\cr
    \hfil\raisebox{0.15ex}{$\m@th#1\scriptscriptstyle+$}\hfil\cr
  }}}%
}
\newcommand{\sqcupplus}{\mathbin{\mathpalette\sqcupplus@inner\relax}}
\makeatother

\newtheorem{lemma}{Lemma}[subsection]
\newtheorem{corollary}[lemma]{Corollary}

\newtheorem{proposition}[lemma]{Proposition} 

\newtheorem{theorem}{Theorem}[section]

\theoremstyle{definition}
\newtheorem{definition}[lemma]{Definition}
\newtheorem{example}[lemma]{Example}

\theoremstyle{remark}

\newtheorem*{remark}{Remark}

\numberwithin{equation}{subsection}

\DeclareMathOperator{\sch}{\mathfrak{S}}
\newcommand{\shup}{{\uparrow}}

\newcommand{\wof}[1]{w_{#1}}

\newcommand{\wtt}{\mathtt{wt}}
\newcommand{\QY}{\ensuremath{\mathbb{QY}}}

\newcommand{\grass}{\mathcal{G}}
\newcommand{\plac}{\mathrm{Plac}}

\newcommand{\rtt}{\mathtt{rt}}
\newcommand{\trm}{\mathtt{trim}}
\newcommand{\zeromap}{\mathcal{Z}}
\newcommand{\hr}{\mathtt{ht}}
\newcommand{\clp}{\mathtt{clip}}
\newcommand{\BRC}{\mathcal{BRC}}
\newcommand{\RC}{\mathcal{RC}}

\newcommand{\ltb}{\mathrm{little}}

\newcommand{\SSYT}{\mathrm{SSYT}}

\newcommand{\coma}{\mathcal{A}}
\newcommand{\dcoma}{\mathcal{D}}

\title{A Littlewood--Richardson rule for forest polynomials}
\author{Matthew J. Samuel}

\newcommand{\dom}{\mathfrak{d}}

\newcommand{\code}{\mathfrak{c}}

\newcommand{\forest}{\mathfrak{P}}
\newcommand{\fcode}{\mathfrak{c}_\forest}
\newcommand{\finv}{P_{\forest}}

\newcommand{\maxd}{\mathfrak{m}}
\newcommand{\ddv}[1]{\delta}

\begin{document}

\begin{abstract}
The forest polynomials $\forest_a$ of Nadeau--Tewari form a $\mathbb{Z}$-basis of $\mathbb{Z}[x_1,x_2,\ldots]$ whose role for the cohomology of the quasisymmetric flag variety parallels that of Schubert polynomials for the classical flag variety. Nonnegativity of the structure constants $\beta_{a,b}^c$ in $\forest_a\,\forest_b = \sum_c \beta_{a,b}^c\,\forest_c$ is known, but no Littlewood--Richardson-style enumerative rule has been available. We give such a rule: $\beta_{a,b}^c$ counts pairs of forest RC graphs of forest-codes $a$ and $b$ whose lift product lands on a forest RC graph of forest-code and weight both equal to $c$. The same rule descends to the cup product on $H^\bullet(\mathrm{QFl}_n)$. The proof introduces a squash product $\boxtimes$ on RC graphs, which on full Grassmannian RC graphs of a fixed height recovers the plactic monoid, a multiplactic algebra $\mathcal{M}_n$ that models Schubert polynomial multiplication, and a lift product $*$ on RC graphs whose quasi-crystal structure, in the framework of Cain--Guilherme--Malheiro, transports the Nadeau--Tewari forest equivalence through the computation. The row-decomposition machinery on RC graphs that underlies these constructions is developed in a companion paper.
\end{abstract}

	\maketitle

\section{Introduction}

The \emph{forest polynomials} $\forest_a$ of Nadeau and Tewari \cite{nadeau2024forest}, indexed by weak compositions $a$, form a $\mathbb{Z}$-basis of the polynomial ring whose role for the cohomology of the quasisymmetric flag variety $\mathrm{QFl}_n$ \cite{bergeron2025qsymflag} parallels that of Schubert polynomials for the classical flag variety $\mathrm{Fl}_n$. The structure constants $\beta_{a,b}^c$ of multiplication of forest polynomials,
$$\forest_a\,\forest_b \;=\; \sum_c \beta_{a,b}^c\,\forest_c,$$
are nonnegative integers; this was first established by Nadeau and Tewari \cite{nadeau2024forest} via the forest equivalence class structure on monomials, and was given a constructive straightening algorithm by Bergeron--Gagnon--Nadeau--Spink--Tewari \cite{bergeron2025equivariant}. A combinatorial Littlewood--Richardson-style rule realizing each $\beta_{a,b}^c$ as the cardinality of an explicit finite set --- in the spirit of the classical LR rule for Schur functions --- has been raised as a natural question in this circle of work \cite{nadeau2024forest,nadeau2024quasisymmetric,bergeron2025qsymflag}, but no such rule has been available.

The main result of this paper is such a rule. The forest-code map $\code_\forest$ partitions the set $\RC_n$ of RC graphs with $n$ rows into Nadeau--Tewari forest equivalence classes whose generating polynomials are the $\forest_a$. Within each such class we single out a distinguished representative: call an RC graph $R$ a \emph{forest RC graph} if its forest-code equals its weight, $\code_\forest(R) = \wtt(R)$. Each forest equivalence class contains a unique forest RC graph, but --- in contrast to the rigid correspondence between RC graphs and Schubert polynomials --- several distinct forest equivalence classes can share the same forest code, so a given weak composition $a$ may be realized by more than one forest RC graph. This nonuniqueness is genuine and visible in computations with the rule below: a forest LR coefficient of $2$ already arises from two distinct forest RC graphs of the same composition $c$ (see Example~\ref{example:forestlrproduct}).

\begin{theorem}[Forest polynomial LR rule]\label{theorem:forestpolyLR}
Let $a,b$ be weak compositions of length $n$. Then
$$\forest_a\,\forest_b \;=\; \sum_c \beta_{a,b}^c\,\forest_c,$$
where $\beta_{a,b}^c$ is the number of pairs $(A,B)$ of RC graphs such that $\code(\wof{A}) = \code_\forest(A) = a$, $\code(\wof{B}) = \code_\forest(B) = b$, and $A*B = R$ for some forest RC graph $R$ with $\code_\forest(R) = \wtt(R) = c$.
\end{theorem}
Here $A*B$ is the \emph{lift product} (\S\ref{section:liftproduct}), a binary operation on RC graphs that we introduce for this rule.

Our rule is not fundamentally different from the shuffle rule of Nadeau and Tewari for products of forest polynomials \cite[\S6]{nadeau2024forest}. It may be viewed as a combinatorialization of that rule: the shuffle rule groups shuffles by their forest-insertion invariant, while our additional RC-graph machinery makes it easier to identify which forest polynomials occur and expresses each coefficient as the cardinality of one explicitly described set. The degree-one case illustrates this advantage (Proposition~\ref{proposition:degreeone}).

Under the Borel-type surjection $\mathbb{Z}[x_1,\ldots,x_n]\twoheadrightarrow \mathbb{Z}[x_1,\ldots,x_n]/\mathrm{QSym}_n^+\cong H^\bullet(\mathrm{QFl}_n)$ of \cite{bergeron2025qsymflag}, the forest polynomials descend to a basis of $H^\bullet(\mathrm{QFl}_n)$, so the same rule computes the cup product there (Corollary~\ref{corollary:cupforest}).

The proof requires substantial machinery, which we develop in its own right because it is of independent interest. We define a \emph{squash product} $\boxtimes$ on RC graphs of a common height, which on full Grassmannian RC graphs of height $n$ coincides with the plactic monoid product on semistandard Young tableaux (Theorem~\ref{theorem:placticiso}), and which makes the free abelian group $\BRC_n$ on $n$-row RC graphs into a right module over the plactic algebra in a way that models the classical structure of $\mathbb{Z}[x_1,\ldots,x_n]$ as a free module over symmetric polynomials with Schubert basis (Theorem~\ref{theorem:schubmoduleiso}). Tensoring the height-stratified plactic algebras together gives rise to a \emph{multiplactic algebra} $\mathcal{M}_n$, which correctly models Schubert polynomial multiplication via a signed expansion (Theorem~\ref{theorem:schubproduct}). The lift product $*$ is then obtained by choosing a canonical elementary tensor representative of each RC graph in $\mathcal{M}_n$; it does \emph{not} model Schubert multiplication, but it is associative, weight-additive, and compatible with the quasi-crystal structure on RC graphs in the framework of Cain--Guilherme--Malheiro \cite{cain2023quasicrystals} (Proposition~\ref{proposition:quasiprod}), and this is exactly enough to transport the Nadeau--Tewari forest equivalence through the computation. As a byproduct, we recover the shuffle product rule for slide polynomials of Assaf--Searles \cite{assaf2017slide} (Theorem~\ref{theorem:slideproduct}).

\subsection*{Relation to the companion paper}

This paper is a companion to \cite{samuelbialgebra}, in which we introduce a bialgebra $\coma=\bigoplus_{n\geq 0}\mathbb{Z}[x_1,\ldots,x_n]$ that we call the \emph{Schubert bialgebra}, study its graded dual $\dcoma$, and prove Littlewood--Richardson rules for the dual Schubert, dual key, dual forest, and dual slide bases of $\dcoma$. The central tool there is a row-decomposition of RC graphs: to an RC graph $R$ with $n$ rows and an integer $0\leq p\leq n$ one associates two RC graphs $\clp^p(R)\in\RC_p$ and $\trm^p(R)\in\RC_{n-p}$, which the reader should picture as cutting $R$ horizontally between rows $p$ and $p+1$. The map $\trm^p$ simply discards the top $p$ rows. The map $\clp^p$ is the subtler partner: it is built by iterating a map $\zeromap:\RC_n^0\to\RC_{n-1}$ on RC graphs whose last row is empty, which realizes the Lascoux--Sch\"utzenberger transition formula on the level of RC graphs by means of Little's bumping algorithm \cite{little2003combinatorial}. The construction preserves crystal operators and forest invariants, and it is these properties, rather than the bialgebra $\coma$ itself, that the present paper uses.

We have collected in \S\ref{section:background} all of the definitions and all of the results from \cite{samuelbialgebra} that are needed here, with precise references to the companion paper for the proofs, so that this paper can be read independently. Conversely, the companion paper does not depend on anything proved here. The first version of the present preprint (arXiv:2606.23876v1, under the title \emph{A Littlewood--Richardson rule for forest polynomials via the Schubert bialgebra}) contained both the Schubert bialgebra and the forest LR rule; the bialgebra material has been split off into \cite{samuelbialgebra} (arXiv:2610.22582), and the present version is a reorganization of what was the final section.

\subsection*{Outline of the paper}

Section~\ref{section:background} (\nameref{section:background}) collects the background on permutations, bounded RC graphs, the maps $\zeromap$, $\clp^p$, $\trm^p$, the Demazure crystal structure on RC graphs, the elementary basis of the polynomial ring, forest polynomials and forest equivalence, and slide polynomials and quasi-Yamanouchi RC graphs.

Section~\ref{section:forestrule} (\nameref{section:forestrule}) contains the main results. In \S\ref{section:multiplactic} we define the squash product $\boxtimes$ and the multiplactic algebra $\mathcal{M}_n$ and prove that it models Schubert multiplication. In \S\ref{section:liftproduct} we define the lift product $*$ and establish its basic properties: weight-additivity and associativity. In \S\ref{section:slidequotient} we equip $\RC_n(w)$ with a seminormal quasi-crystal structure and prove that $*$ is a quasi-crystal isomorphism onto its image, which yields a product rule for slide polynomials. In \S\ref{section:forestlrsub} we prove Theorem~\ref{theorem:forestpolyLR} and work out an example in which a coefficient $2$ arises from two distinct forest RC graphs. Finally, \S\ref{subsection:qflbranch} records the geometric interpretation as the cup product on the quasisymmetric flag variety.

A list of notation appears at the end of the paper.

\section{Background}\label{section:background}

This section collects the definitions and results we need, almost all of which are taken from the companion paper \cite{samuelbialgebra}, to which we refer for proofs. Nothing in this section is new, with the exception of a few pieces of notation (such as $\mathrm{Schub}_n$ in \S\ref{section:brc}) that are convenient for what follows.

\subsection{Permutations}\label{section:perms}

Let $S_\infty$ denote the group of permutations of $\mathbb{Z}_{>0}$ fixing all but finitely many elements, and let $s_i=(i,i+1)$ denote the simple transpositions. For $w\in S_\infty$ we write $\ell(w)$ for the Coxeter length, which is the number of inversions, and $t_{ab}$ for the transposition exchanging $a$ and $b$. We identify $S_n$ with the subgroup of $S_\infty$ fixing all $i>n$.

\begin{definition}[Lehmer code, dual code, dominant permutations]\label{definition:codestuff}
For a permutation $w\in S_\infty$, the \emph{(right) Lehmer code} of $w$ is the sequence $\code(w)=(c_1,c_2,\ldots)$ defined by
$$c_i \;=\; \#\{\,j>i : w(j)<w(i)\,\}.$$
This is a weak composition with $\sum_i c_i=\ell(w)$, and the map $w\mapsto\code(w)$ is a bijection from $S_\infty$ onto the set of weak compositions of finite support. We will sometimes write $\code_i(w):=c_i$. The \emph{dual code} of $w$ is
$$\code^*(w) \;:=\; \code(w^{-1}).$$

A permutation $\mu\in S_\infty$ is called \emph{dominant} if $\code(\mu)$ is a partition (a weakly decreasing sequence of nonnegative integers, padded by $0$s). For a partition $\lambda$, we write $\dom_\lambda$ for the dominant permutation with
$$\code^*(\dom_\lambda) \;=\; \lambda.$$
\end{definition}

\begin{definition}[Upward shifting, maximum right descent, Grassmannian permutations]
For a permutation $v\in S_\infty$, we define $\shup{v}$ to be the permutation defined by
$$\shup{v}(i) = \begin{cases}
	v(i-1)+1&\mbox{ if }i>1\\
	1&\mbox{ if }i=1
\end{cases}$$
Recursively, we define $\shup^k(v)$ to be $\shup(\shup^{k-1}(v))$. In the literature, this is sometimes written as $\shup^kv = 1^k\times v$. For a permutation $u$ and $N\geq 0$ we write $u\shup^N v$ for the product $u\cdot\shup^N v$ in $S_\infty$.

For a permutation $w$, we define $\maxd(w)$ to be the maximum right descent of $w$. That is,
$$\maxd(w) = \max\{i\mid w(i)>w(i+1)\},$$
with $\maxd(\mathrm{id})=0$. A permutation $v$ is \emph{$k$-Grassmannian} if it has at most one descent, and that descent, if it exists, is at position $k$. A $k$-Grassmannian permutation $v$ is determined by its \emph{shape}, the partition $\mathrm{shape}(v)$ obtained by sorting $\code(v)$ into weakly decreasing order; explicitly $\code(v) = (\lambda_k,\lambda_{k-1},\ldots,\lambda_1,0,\ldots)$ where $\lambda=\mathrm{shape}(v)$, and we say $\mathrm{shape}(v)$ has at most $k$ parts.
\end{definition}

\begin{definition}[Schubert polynomials and the elementary basis]\label{definition:schubert}
We write $\sch_w(x)=\sch_w(x_1,x_2,\ldots)$ for the Schubert polynomial of $w\in S_\infty$ \cite{notes}, and for $n\geq\maxd(w)$ we write $\sch_w^n$ for $\sch_w$ regarded as an element of $\coma_n := \mathbb{Z}[x_1,\ldots,x_n]$, which is possible because $\sch_w$ involves only the variables $x_1,\ldots,x_{\maxd(w)}$. The Schubert polynomials $\sch_w^n$, $\maxd(w)\leq n$, form a $\mathbb{Z}$-basis of $\coma_n$, and the Schubert structure constants $c_{u,v}^w$ are defined by
$$\sch_u\sch_v = \sum_w c_{u,v}^w\sch_w.$$
We write $\Lambda_n\subseteq\coma_n$ for the ring of symmetric polynomials in $x_1,\ldots,x_n$, $e_{p,k}^n=e_p(x_1,\ldots,x_k)\in\coma_n$ for the elementary symmetric polynomial of degree $p$ in the first $k\leq n$ variables, and $s_\lambda(x_1,\ldots,x_n)$ for the Schur polynomial. Recall that if $v$ is $n$-Grassmannian then $\sch_v^n = s_{\mathrm{shape}(v)}(x_1,\ldots,x_n)$, and that $\coma_n$ is a free $\Lambda_n$-module with basis $\{\sch_w^n\mid w\in S_n\}$ \cite{notes}.

An \emph{$n$-elementary weak composition} is a finite sequence of nonnegative integers $\alpha=(\alpha_1,\alpha_2,\ldots,\alpha_N)$ such that $\alpha_i\leq i$ for all $i<n$ and $\alpha_i\leq n$  for all $i\geq n$, satisfying the additional restriction that $\alpha_i\geq \alpha_{i+1}$ if $i\geq n$. An \emph{$n$-elementary monomial} is an element of $\coma_n$ of the form 
$$e_\alpha^n := e_{\alpha_1,1}^{1}\cdots e_{\alpha_n,n}^ne_{\alpha_{n+1},n}^n\cdots e_{\alpha_N,n}^n$$
where $\alpha$ is an $n$-elementary weak composition. The set $\mathcal{E}_n$ of $n$-elementary monomials is a $\mathbb{Z}$-basis of $\coma_n$ \cite[\S\ref*{B-section:elementary}]{samuelbialgebra}, and we define integers $E_w^{\alpha;n}$ for $\maxd(w)\leq n$ by
$$\sch_w^n = \sum_\alpha E_w^{\alpha;n}\,e_\alpha^n.$$
\end{definition}

The companion paper gives an explicit signed formula for the $E_w^{\alpha;n}$ in terms of chains in Bruhat order \cite[Theorem~\ref*{B-theorem:elemtrans}]{samuelbialgebra}; here we only need that they exist and are integers, though we note that they genuinely have mixed signs in general (see \cite{woodruff2025single} for a classification of the permutations for which $\sch_w$ is a single elementary monomial). We also record two classical facts about double Schubert polynomials $\sch_w(x;y)$ that are used in \S\ref{section:liftproduct}: for a partition $\lambda$ with $N=\ell(\lambda)$, writing $E_{p,p}(x;-y)=\prod_{i=1}^p(x_i+y)$, one has the factorization
$$\sch_{\dom_\lambda}(x;-y) = E_{\lambda_1,\lambda_1}(x;-y_1)\cdots E_{\lambda_N,\lambda_N}(x;-y_N)$$
\cite[(6.14)]{notes}, and the Cauchy formula
$$\sch_w(x;-y) = \sum_{u} \sch_u(x)\,\sch_{uw^{-1}}(y),$$
where the sum is over $u$ with $\ell(uw^{-1})=\ell(w)-\ell(u)$ \cite{notes}.

\subsection{Bounded RC graphs}\label{section:brc}

\begin{definition}[{\cite[Definition~\ref*{B-definition:rcgraph}]{samuelbialgebra}}]\label{definition:rcgraph}
	Let $R\subseteq \mathbb{P}\times \mathbb{P}$ be a finite set, where $\mathbb{P}=\mathbb{Z}_{>0}$. To each such set we associate an element $\wof{R}$ of $S_\infty$ as follows. Given a pair $(i, j)$ where $i,j>0$ are integers, define $s(i,j) = s_{i+j - 1}$. Totally order the grid such that $(i, j) < (a, b)$ if and only if $i < a$ or $i = a$ and $j > b$ (in other words, lexicographical order, except that the order on the second coordinate is reversed). By this ordering, index $R$ as $r_1, r_2,\ldots, r_m$ in increasing order. Then
	$$\wof{R} = s_{r_1}\cdots s_{r_m}.$$
	If $\ell(\wof{R}) = m$, then we say that $R$ is an \emph{RC graph}, and we call the reduced word $\mathrm{word}(R)=(i_1+j_1-1,\ldots,i_m+j_m-1)$, where $r_k=(i_k,j_k)$, the \emph{reading word} of $R$. Thus the reading word is read row by row from top to bottom, and from right to left within each row.
	
	A \emph{bounded RC graph} is an RC graph $R$ together with a specified number of rows $n\geq \max\{i: (i,j)\in R\text{ for some }j\}$, with the added condition that $\maxd(\wof{R})\leq n$. The definition is as an ordered pair $(R, n)$, but it will be far more convenient to abuse notation and consider the number of rows a property of $R$ itself. To denote the number of rows, we define the \emph{height} of $R$ to be $\hr(R)=n$. A bounded RC graph has an associated weight vector $\wtt(R)\in\mathbb{Z}_{\geq 0}^n$ where $\wtt_i(R)$ is the number of elements in row $i$. We write $\RC_n$ for the set of bounded RC graphs of height $n$, $\RC_n(w)$ for those with $\wof{R}=w$, and $\RC_n^0$ (resp.\ $\RC(w,n)^0$) for the subset of $\RC_n$ (resp.\ $\RC_n(w)$) consisting of RC graphs whose last row is empty.
	
	To a bounded RC graph $R$ with $\hr(R)\geq 1$, we define the \emph{trim} operation $\trm(R)$ to be the bounded RC graph with height $\hr(R)-1$ and underlying set $\{(i-1, j): (i,j)\in R, i\geq 2\}$. We also define
	$$\shup^m R=\{(i+m, j)| (i,j)\in R\}.$$
	For $w$ with $\maxd(w)\leq n$, the \emph{bottom RC graph} $\mathrm{bottom}(w)\in\RC_n(w)$ is the RC graph $\{(i,j): 1\leq j\leq \code_i(w)\}$, whose weight is $\code(w)$.
\end{definition}

That $\trm(R)$ is again a bounded RC graph, i.e.\ that $\maxd(\wof{\trm R})\leq\hr(R)-1$, is \cite[Lemma~\ref*{B-L:trim}]{samuelbialgebra}. We use freely the fact that the row-weight generating function of RC graphs is the Schubert polynomial \cite{bjs,fominstan}: for $\maxd(w)\leq n$,
\begin{equation}\label{eq:bjs}
\sch_w^n(x) = \sum_{R\in\RC_n(w)} x^{\wtt(R)}.
\end{equation}
The reduced words of $w$ paired with their \emph{compatible sequences} in the sense of \cite{bjs} are precisely the pairs $(\mathrm{word}(R),(i_1,\ldots,i_m))$ with $R\in\RC_n(w)$, where $i_k$ is the row of the $k$-th crossing in reading order; thus the elements of $\RC_n(w)$ with a given reading word $\mathbf r$ are in bijection with the compatible sequences of $\mathbf r$ with entries at most $n$.

Let $\BRC$ be the free abelian group spanned by all bounded RC graphs. If $\BRC_n$ is the subgroup spanned by all bounded RC graphs $R$ with $\hr(R)=n$, then we have a grading
$$\BRC = \bigoplus_{n=0}^\infty \BRC_n.$$
There is an evident evaluation map $\mathrm{ev}: \BRC_n\to \coma_n$ defined on basis elements as
$$\mathrm{ev}(R) = x^{\wtt(R)},$$
which is a surjective homomorphism of abelian groups. We define elements $\mathcal{S}_w(n)\in\BRC_n$ for $\maxd(w)\leq n$ by
$$\mathcal{S}_w(n) = \sum_{R\in\RC_n(w)}R,$$
and we define $\mathrm{Schub}_n\subseteq\BRC_n$ to be the $\mathbb{Z}$-submodule spanned by the $\mathcal{S}_w(n)$, $\maxd(w)\leq n$. By \eqref{eq:bjs}, $\mathrm{ev}(\mathcal{S}_w(n))=\sch_w^n$, so $\mathrm{ev}$ restricts to an isomorphism of abelian groups $\mathrm{Schub}_n\to\coma_n$.

\begin{definition}[Roots]\label{definition:roots}
Given a pair of positive integers $i,j$ and an RC graph $R$, define 
$$R[i, j]=\{(a, b)\in R\mid (a, b) > (i, j)\}$$
in the total order of Definition~\ref{definition:rcgraph}. The \emph{right root} of the position $(i,j)$ in $R$ is
$$\rtt_R(i,j)=(\wof{R[i, j]}^{-1}(i + j - 1), \wof{R[i,j]}^{-1}(i+j)).$$
If $(i,j)\in R$ then $\rtt_R(i,j)$ is a pair $(a,b)$ with $a<b$, and as $(i,j)$ ranges over $R$ the right roots range bijectively over the inversions of $\wof{R}$. We write $\Phi^-$ for the set of pairs $(a,b)$ with $a>b$ (negative roots); the condition $\rtt_R(i,j)\in\Phi^-$ for some $(i,j)\in R$ is equivalent to $R$ not being an RC graph.
\end{definition}

There is a common visualization of RC graphs as \emph{pipe dreams}. In this visualization, we draw a grid, and in each position $(i,j)$ we draw either a crossing (if $(i,j)\in R$) or an elbow (if $(i,j)\notin R$). Then we draw pipes entering from the left edge of the grid, with the pipe entering at row $i$ labeled $i$. The pipes travel through the grid, turning at elbows and crossing at crossings, and exit at the top of the grid. For a bounded RC graph $R$ with $\hr(R)=n$, we only draw the first $n$ rows. In the figures of this paper, we also write at each crossing $(i,j)$ the index $i+j-1$ of the corresponding simple reflection. For a crossing $c\in R$, the right root $\rtt_R(c)=(a,b)$ records the labels of the two pipes crossing at $c$.

\subsection{The transition map and clipping}\label{section:cliptrim}

The map $\trm$ deletes the first row of an RC graph. The companion paper constructs a far subtler map $\zeromap$ that deletes an \emph{empty last row} while preserving the row weights, and shows that it realizes the Lascoux--Sch\"utzenberger transition formula for Schubert polynomials on the level of RC graphs. It is defined by means of Little's bumping algorithm \cite{little2003combinatorial}.

\begin{definition}[Little bumps, {\cite[Definition~\ref*{B-definition:littlebump}]{samuelbialgebra}}]\label{definition:littlebump}
  Given a reduced word $a_1\cdots a_m$ for a permutation $w$ and an index $1\leq p\leq m$, the \emph{Little bump} $\ltb_p(a_1\cdots a_m)$ is defined recursively as follows. If $a_p>1$, replace $a_p$ with $a_p-1$, and if $a_p=1$ then instead replace all $a_q$ with $q\neq p$ with $a_q+1$. Denote the new word by $a_1'\cdots a_m'$. If this resulting word is reduced, we are done. Otherwise, there is a unique index $j\neq p$ such that the word obtained by deleting $a_j'$ from the new word is reduced by the deletion property of Coxeter groups. We then define
  $$\ltb_p(a_1\cdots a_m) = \ltb_j(a_1'\cdots a_m').$$
\end{definition}

On the level of RC graphs, if $c\in X$ is a crossing of an RC graph $X$ such that $X\setminus\{c\}$ is reduced, the \emph{Little bump of $X$ at $c$} is computed by moving $c$ to the nearest elbow to its left in the same row, and whenever this makes two pipes cross twice, moving the other crossing of that pair of pipes to the nearest elbow to its left in its row, until the result is reduced; in terms of reading words, this is $\ltb_c$ applied at the position of $c$, with the row of every crossing unchanged \cite[\S\ref*{B-section:bpd}]{samuelbialgebra}. (In the case that a crossing to be moved has no elbow to its left, it is removed and the process stops; this case does not arise in Definition~\ref{definition:zeromap}.)

\begin{definition}[{\cite[Definition~\ref*{B-definition:zeromap}]{samuelbialgebra}}]\label{definition:zeromap}
Let $R$ be a bounded RC graph of height $n\geq 1$ whose row $n$ is empty, and let $w=\wof R$. If $\maxd(w)<n$, let $\zeromap(R)$ be $R$ with its last row deleted, a bounded RC graph of height $n-1$. If $\maxd(w)=n$, let $s=\max\{j>n: w(j)<w(n)\}$ and let $c\in R$ be the unique crossing with $\rtt_R(c)=(n,s)$. Let $R^{(1)}$ be the RC graph obtained from $R$ by the Little bump at $c$, and define $\zeromap(R)=\zeromap(R^{(1)})$, where $R^{(1)}$ is regarded as a bounded RC graph of height $n$.
\end{definition}

The companion paper proves that this is well posed: the crossing $c$ exists and is unique, $R^{(1)}$ is a bounded RC graph of height $n$ with row $n$ empty and $\wtt(R^{(1)})=\wtt(R)$, and the recursion terminates \cite[Lemma~\ref*{B-lemma:zeroweight}]{samuelbialgebra}. Thus $\zeromap:\RC_n^0\to\RC_{n-1}$ is a weight-preserving map. For a permutation $w$ with $\maxd(w)\leq n$, let $T_n(w)$ be the set of permutations $w'$ such that $\sch_{w'}$ occurs in the Schubert expansion of $\sch_w(x_1,\ldots,x_{n-1},0)$. Each $w'\in T_n(w)$ has $\ell(w')=\ell(w)$ and $\maxd(w')<n$, and in fact
\begin{equation}\label{eq:Tn}
\sch_w(x_1,\ldots,x_{n-1},0) = \sum_{w'\in T_n(w)}\sch_{w'}(x_1,\ldots,x_{n-1}),
\end{equation}
with every Schubert polynomial on the right occurring with multiplicity one \cite[(\ref*{B-E:Tn})]{samuelbialgebra}; this is the Lascoux--Sch\"utzenberger transition formula \cite[(4.16)]{notes} iterated until the variable $x_n$ disappears.

\begin{theorem}[{\cite[Theorem~\ref*{B-theorem:zero_bijection}]{samuelbialgebra}}]\label{theorem:zero_bijection}
Let $w$ be a permutation with $\maxd(w)\leq n$. Then
$$\zeromap:\RC(w, n)^{0}\to \bigsqcup_{w'\in T_n(w)}\RC_{n-1}(w')$$
is a weight-preserving bijection.
\end{theorem}

In particular, if $w$ is $n$-Grassmannian of shape $\lambda$, then by the stability of Schur polynomials $T_n(w)$ is empty if $\ell(\lambda)=n$ (in which case $\RC(w,n)^0$ is empty as well) and otherwise $T_n(w)=\{w'\}$, where $w'$ is the $(n-1)$-Grassmannian permutation of shape $\lambda$; in the latter case $\zeromap$ restricts to a weight-preserving bijection $\RC(w,n)^0\to\RC_{n-1}(w')$.

\begin{definition}[{\cite[Definition~\ref*{B-definition:cliptrim}]{samuelbialgebra}}]\label{definition:cliptrim}
Let $R$ be a bounded RC graph with $\hr(R)=n$ and let $0\leq p\leq n$. For $0\leq k\leq n$ write $R_{\leq k}=\{(i,j)\in R: i\leq k\}$.
\begin{enumerate}
\item $\trm^p(R)=\trm(\trm(\cdots\trm(R)))$ ($p$ times) is the bounded RC graph of height $n-p$ with underlying set $\{(i-p,j):(i,j)\in R,\ i>p\}$.
\item Let $h=\max\bigl(p,\maxd(\wof{R_{\leq p}})\bigr)$ and regard $R_{\leq p}$ as a bounded RC graph of height $h$; its rows $p+1,\ldots,h$ are empty. Then
$$\clp^p(R)=\zeromap^{\,h-p}(R_{\leq p})\in\RC_p.$$
\end{enumerate}
\end{definition}

Since $\zeromap$ preserves weights, the last row is empty at every intermediate step, so $\clp^p(R)$ is a bounded RC graph of height $p$ with $\wtt(\clp^p(R))=(\wtt_1(R),\ldots,\wtt_p(R))$. It depends only on $R_{\leq p}$, and $\clp^p(R)=\zeromap^{\,N-p}(R_{\leq p})$ for every $N\geq h$, since $\zeromap$ simply deletes the empty last row whenever the height exceeds $\maxd$. In particular $\clp^p(R)=R_{\leq p}$ whenever $\maxd(\wof{R_{\leq p}})\leq p$. We use the same formula to define $\clp^p(X)$ for any bounded RC graph $X$ of height $\geq p$, and we will also use $\clp^p$ in the opposite direction as a \emph{resizing} operation: a bounded RC graph $R$ of height $n$ with $\maxd(\wof R)\leq n$ may be regarded as a bounded RC graph of any height $N\geq n$ by appending empty rows, and then $\clp^n$ of the result is $R$ again. Thus, when we speak of ``increasing the height'' of an RC graph we mean appending empty rows, and ``clipping to height $p$'' means applying $\clp^p$.

The companion paper uses $\clp^p$ and $\trm^p$ to define an associative product $\sqcupplus$ on $\BRC$,
$$R_1\sqcupplus R_2 = \sum_{\substack{\clp^m(R')=R_1\\ \trm^m(R')=R_2}}R'\qquad(\hr(R_1)=m),$$
and shows that $\dcoma$ is a quotient of $(\BRC,\sqcupplus)$; this is the source of the dual LR rules proved there. In the present paper we do not use $\sqcupplus$, but we do use a different concatenation, the horizontal \emph{disjoint union} $R_1\sqcup R_2$ of \S\ref{section:multiplactic}, followed by $\clp$.

\subsection{The crystal structure on RC graphs}\label{section:crystalbackground}

\begin{definition}[Crystal action on RC graphs, {\cite[Definition~\ref*{B-definition:crystal}]{samuelbialgebra}}] \label{definition:crystal}
Following Morse and Schilling \cite{morse2016crystal}, who defined an $A_{\ell-1}$-crystal structure on reduced factorizations of a permutation $w$, and the adaptation to RC graphs by Assaf and Schilling \cite{assaf2018demazure}, we equip $\RC_n(w)$ with Kashiwara raising and lowering operators $e_i, f_i$ for $1\leq i\leq n-1$. The operator pair $(e_i, f_i)$ acts only on rows $i$ and $i+1$ of an RC graph $R$, viewed as adjacent factors in the corresponding reduced factorization.

The action is governed by a bracketing procedure on the column indices appearing in rows $i$ and $i+1$. Process the entries of row $i$ in decreasing order of column index: for each column $b$ in row $i$, pair it with the smallest column $a>b$ appearing in row $i+1$ that has not yet been paired in an earlier step; if no such $a$ exists, $b$ is left unpaired. Let $R_i(R)$ denote the set of unpaired column indices in row $i$, and $L_i(R)$ the set of unpaired column indices in row $i+1$.

If $R_i(R)$ is nonempty, $f_i(R)$ is obtained by removing the smallest unpaired entry $b\in R_i(R)$ from row $i$ and inserting an entry into row $i+1$ at the largest column $b' < b$ such that $b'$ does not already appear in row $i+1$; if $R_i(R) = \emptyset$ or no such $b'$ exists, $f_i(R) = \emptyset$ (undefined). Dually, if $L_i(R)$ is nonempty, $e_i(R)$ is obtained by removing the largest unpaired entry $a\in L_i(R)$ from row $i+1$ and inserting an entry into row $i$ at the smallest column $a' > a$ such that $a'$ does not already appear in row $i$; if $L_i(R) = \emptyset$ or no such $a'$ exists, $e_i(R) = \emptyset$. We set $\varphi_i(R)=|R_i(R)|$ and $\varepsilon_i(R)=|L_i(R)|$, so that $\varphi_i(R)=\max\{k: f_i^k(R)\neq\emptyset\}$ and $\varepsilon_i(R)=\max\{k:e_i^k(R)\neq\emptyset\}$.

The \emph{weight} of $R$ as an element of the crystal is $\wtt(R)$, and $\wtt(f_i(R))=\wtt(R)-\alpha_i$ where $\alpha_i=\mathbf{e}_i-\mathbf{e}_{i+1}$. The operators $e_i,f_i$ preserve $\wof{R}$ but in general change the reading word. A \emph{highest weight} element is an RC graph $R$ with $e_i(R) = \emptyset$ for all $i$. Every $R\in \RC_n(w)$ corresponds to a unique highest weight element $\mathbb{Y}(R)$, which can be obtained by greedily applying available $e_i$ until none apply. The \emph{Demazure crystal} $\mathrm{Dem}(R)$ generated by $\mathbb{Y}(R)$ is the closure of $\{\mathbb{Y}(R)\}$ under the operators $\{f_i^k : i\geq 1, k\geq 0\}$ subject to the Demazure truncation rule of \cite{assaf2018demazure}; it is the connected component of $R$ in $\RC_n(w)$, and by \cite{assaf2018demazure} it is isomorphic to a Demazure crystal in the sense of Kashiwara \cite{kashiwara1993crystal} and Littelmann \cite{littelmann1995crystal}, with character a key polynomial.
\end{definition}

When $v$ is $n$-Grassmannian, $\RC_n(v)$ is a connected crystal isomorphic to the full Kashiwara crystal $B(\lambda)$ of semistandard Young tableaux of shape $\lambda=\mathrm{shape}(v)$ with entries in $[n]$ (see \S\ref{section:multiplactic} for the explicit bijection).

We use the \emph{tensor product} of crystals in Kashiwara's convention: for crystals $B_1,B_2$ and $b_1\otimes b_2\in B_1\otimes B_2$,
\begin{equation}\label{eq:tensorrule}
f_i(b_1\otimes b_2) = \begin{cases} f_i(b_1)\otimes b_2 & \text{if }\varphi_i(b_1)>\varepsilon_i(b_2),\\ b_1\otimes f_i(b_2) & \text{otherwise,}\end{cases}
\qquad
e_i(b_1\otimes b_2) = \begin{cases} e_i(b_1)\otimes b_2 & \text{if }\varphi_i(b_1)\geq\varepsilon_i(b_2),\\ b_1\otimes e_i(b_2) & \text{otherwise.}\end{cases}
\end{equation}
With this convention, the plactic product $T_1\cdot T_2$ of semistandard tableaux (concatenation of reading words followed by Knuth equivalence, or equivalently RSK/rectification) is a crystal isomorphism from the connected component of $T_1\otimes T_2$ in $B(\mu_1)\otimes B(\mu_2)$ onto $B(\nu)$, $\nu=\mathrm{shape}(T_1\cdot T_2)$ \cite{kashiwara1993crystal,littelmann1995crystal}. For tensor products of Demazure crystals we use the results of Kouno \cite{kouno2020decomposition} on their decomposition into Demazure crystals; the precise statements needed are recalled where they are used.

The crucial compatibility between the crystal structure and the map $\zeromap$ is the following.

\begin{proposition}[{\cite[Corollary~\ref*{B-corollary:crystaliso}]{samuelbialgebra}}] \label{proposition:zerocrystal}
For any bounded RC graph $R\in\RC_n^0$ and any $1\leq i\leq n-2$ we have
$$\zeromap(e_i(R)) = e_i(\zeromap(R))\qquad\text{and}\qquad \zeromap(f_i(R)) = f_i(\zeromap(R))$$
whenever the left-hand sides are defined, and $e_i(\zeromap R)$ (resp.\ $f_i(\zeromap R)$) is defined if and only if $e_i(R)$ (resp.\ $f_i(R)$) is. Consequently $\zeromap$ restricts to an isomorphism of crystals
$$\RC(w,n)^0\;\longrightarrow\;\bigsqcup_{w'\in T_n(w)}\RC_{n-1}(w'),$$
with operator indices $1\leq i\leq n-2$, and for any $p$ the map $\clp^p$ commutes with $e_i$ and $f_i$ for $1\leq i\leq p-1$.
\end{proposition}

The proof in the companion paper proceeds by showing that Edelman--Greene insertion identifies the crystal operators on RC graphs with tableau crystal operators on the recording tableau, and then invoking the theorem of Hamaker and Young \cite[Theorem~4.4]{hamaker2014relating} that Little bumps preserve the Edelman--Greene recording tableau.

\subsection{Forest polynomials and forest equivalence}\label{section:forestbackground}

The \emph{forest polynomials} $\forest_a$ were introduced by Nadeau and Tewari in \cite{nadeau2024forest}. Like Schubert polynomials, they have a positive formula in terms of reduced words and compatible sequences, but the invariant preserved by the relevant equivalence relation on reduced words is an indexed forest rather than a permutation.

\begin{definition}[Indexed forests and codes]
  An \emph{indexed forest} $F$ is a pair $(S, T_\bullet)$ where $S$ is a finite set of positive integers, called the \emph{support} of $F$. Writing
  $$S=I_1\cup I_2\cup \cdots \cup I_k$$
  where $I_1,I_2,\ldots, I_k$ are the maximal intervals contained in $S$, we have $T_\bullet$ as a sequence of binary trees $T_1,T_2,\ldots, T_k$ where $T_j$ has $|I_j|$ internal nodes for each $j$. The internal nodes of $F$ are denoted by $\mathrm{IN}(F)$. For $v\in \mathrm{IN}(F)$, we denote by $\mathrm{left}(v)$ and $\mathrm{right}(v)$ the left and right children of $v$ respectively, which may not be in $\mathrm{IN}(F)$ if they are leaves.
  
  There is a natural labeling $\lambda_F:\mathrm{IN}(F)\to S$ sending each internal node, say $v\in T_j$, to the unique element of $I_j$ that is at the same index as $v\in T_j$ in the inorder traversal of $T_j$.
  
  The function $\rho_F:\mathrm{IN}(F)\to S$ is defined for an internal node $v\in T_j$ by 
  $$\rho_F(v) = \begin{cases}
  \lambda_F(v)&\mbox{ if }\mathrm{left}(v)\notin\mathrm{IN}(F)\\
  \lambda_F(\mathrm{left}(v))&\mbox{ if }\mathrm{left}(v)\in\mathrm{IN}(F).
  \end{cases}$$
  For each indexed forest $F$, we assign a weak composition $\code(F)$ by
  $$\code_j(F)=\#\{v\in \mathrm{IN}(F): \rho_F(v)=j\}.$$
  Then the \emph{forest polynomial} $\forest_F(x)$ is defined by
  $$\forest_F(x) = \sum_{\tau}\prod_{v\in\mathrm{IN}(F)}x_{\tau(v)}$$
  where the sum is over all functions $\tau:\mathrm{IN}(F)\to \mathbb{Z}_{>0}$ such that $\tau(v)\leq \rho_F(v)$ for all $v\in \mathrm{IN}(F)$, $\tau(v)\leq\tau(\mathrm{left}(v))$ for all $v\in \mathrm{IN}(F)$ such that $\mathrm{left}(v)\in \mathrm{IN}(F)$, and $\tau(v)<\tau(\mathrm{right}(v))$ for all $v\in \mathrm{IN}(F)$ such that $\mathrm{right}(v)\in \mathrm{IN}(F)$.

  The map $F\mapsto\code(F)$ is a bijection from indexed forests to weak compositions of finite support \cite{nadeau2024forest}, so we may index forest polynomials by weak compositions: $\forest_a(x) := \forest_F(x)$ where $F$ is the unique indexed forest such that $\code(F) = a$. The forest polynomials $\forest_a$ form a $\mathbb{Z}$-basis of $\mathbb{Z}[x_1,x_2,\ldots]$, and those with $a$ of length at most $n$ form a $\mathbb{Z}$-basis of $\coma_n$ \cite{nadeau2024forest}.
\end{definition}

\begin{definition}[Injective words, forest invariant]\label{definition:forestinvariant}
In \cite{nadeau2024ppartitions}, an alphabet $\overline{\mathbb{Z}}$ is defined as the set of pairs of integers $(i,j)$ in lexicographical order, written with the notation $i^{[j]}$. They define $\mathrm{val}(i^{[j]})=i$. An \emph{injective word} is a word $w=w_1w_2\cdots w_k$ such that $w_i\neq w_j$ for all $i\neq j$. An \emph{LBS (local binary search) labeling} of an indexed forest $F$ is an injective function $\rho:\mathrm{IN}(F)\to \overline{\mathbb{Z}}$ such that
\begin{enumerate}
\item $\rho(\mathrm{left}(v))<\rho(v)$ for all $v\in \mathrm{IN}(F)$ such that $\mathrm{left}(v)\in \mathrm{IN}(F)$.
\item $\rho(v)<\rho(\mathrm{right}(v))$ for all $v\in \mathrm{IN}(F)$ such that $\mathrm{right}(v)\in \mathrm{IN}(F)$.
\item For all $v\in\mathrm{IN}(F)$, $\mathrm{val}(\rho(v)) = \lambda_F(u)$ for some $u\in\mathrm{IN}(F)$.
\end{enumerate}

Following \cite{nadeau2024forest}, given any injective word $w = w_1\cdots w_m$ in $\overline{\mathbb{Z}}$, we associate to it inductively an indexed forest $F(w)$, an LBS labeling $P(w)$ of $F(w)$, and an injective ``time-stamp'' labeling $Q(w):\mathrm{IN}(F(w))\to\{1,\ldots,m\}$ as follows. If $m=0$, then $F(w)$ is the empty forest. Otherwise write $w = w'a$ with $w' = w_1\cdots w_{m-1}$, and assume $F' := F(w')$, $P' := P(w')$, and $Q' := Q(w')$ have already been constructed; let $S'$ be the support of $F'$, and let $S' = I_1\cup\cdots\cup I_k$ be its decomposition into maximal intervals (listed left to right). For each $j$, denote by $u_j$ the root of the tree of $F'$ supported on $I_j$, and set $a_j := P'(u_j)$. We obtain $F(w)$ by adjoining a new internal node $u$ to $F'$, with attachments determined by $i := \mathrm{val}(a)$:
\begin{itemize}
\item If $i\notin S'$, then $u$ is the root of a new tree, with left child $u_j$ when $i-1\in I_j$ (otherwise no left child), and right child $u_j$ when $i+1\in I_j$ (otherwise no right child).
\item If $i\in I_j$, then we compare $a$ with $a_j$ in the order on $\overline{\mathbb{Z}}$:
\begin{itemize}
\item if $a>a_j$, then $u_j$ becomes the left child of $u$; additionally, if $\max(I_j)+2\in S'$ (necessarily $\max(I_j)+2 = \min(I_{j+1})$), then $u_{j+1}$ becomes the right child of $u$;
\item if $a<a_j$, then $u_j$ becomes the right child of $u$; additionally, if $\min(I_j)-2\in S'$ (necessarily $\min(I_j)-2=\max(I_{j-1})$), then $u_{j-1}$ becomes the left child of $u$.
\end{itemize}
\end{itemize}
The remaining roots of $F'$ become roots of $F(w)$ unchanged. We then extend $P'$ to $P(w)$ by setting $P(w)(u) := a$, and extend $Q'$ to $Q(w)$ by setting $Q(w)(u) := m$. The pair $(P(w), Q(w))$ is the \emph{$\Omega$-insertion} of $w$. By \cite[Theorem~5.1]{nadeau2024forest}, $w\mapsto(P(w),Q(w))$ is a bijection from injective words in $\overline{\mathbb{Z}}$ onto pairs $(P,Q)$ of an LBS labeling and a decreasing labeling of a common indexed forest.

We write $\mathsf{Sh}(W_1,W_2)$ for the multiset of \emph{shuffles} of two injective words $W_1,W_2$ with disjoint letter sets, i.e.\ the words containing $W_1$ and $W_2$ as complementary subwords. The \emph{forest class} of an injective word $W$ is the set of injective words $W'$ with $P(W')=P(W)$; for forest classes $\Phi_1,\Phi_2$ we write $\mathsf{Sh}(\Phi_1,\Phi_2)=\bigsqcup_{W_1\in\Phi_1,W_2\in\Phi_2}\mathsf{Sh}(W_1,W_2)$.
\end{definition}

\begin{definition}[Injectification, $\equiv_\forest$, forest codes]\label{definition:forestequiv}
  For a word $\mathbf{r}$ in the alphabet $\mathbb{Z}$, we may canonically turn $\mathbf{r}$ into an injective word $\overline{\mathbf r}$ in the alphabet $\overline{\mathbb{Z}}$ by assigning superscripts consecutively starting at $1$ to break ties between equal letters.

For words $\mathbf{r}_1$ and $\mathbf{r}_2$ we say that $\mathbf{r}_1\equiv_\forest\mathbf{r}_2$ if $P(\overline{\mathbf{r}_1}) = P(\overline{\mathbf{r}_2})$, and we write $\finv(\mathbf r)=P(\overline{\mathbf r})$. For an RC graph $R$ we define its \emph{forest word} to be $\mathrm{rev}(\mathrm{word}(R))$, the reversal of the reading word, and its \emph{injectified word} to be
$$\overline{R} := \overline{\mathrm{rev}(\mathrm{word}(R))}.$$
For RC graphs $R_1$ and $R_2$, we say that $R_1\equiv_\forest R_2$ if $\mathrm{rev}(\mathrm{word}(R_1))\equiv_\forest \mathrm{rev}(\mathrm{word}(R_2))$, i.e.\ if $P(\overline{R_1})=P(\overline{R_2})$. The reversal is forced by the conventions: the reduced words of \cite{nadeau2024forest} are read in the order opposite to our reading order, and Theorem~\ref{theorem:forestformula} below fails for the unreversed reading word already for $w=2413$ \cite[Definition~\ref*{B-definition:forestequiv}]{samuelbialgebra}. Since $\equiv_\forest$ is defined via the reading word, $R_1\equiv_\forest R_2$ implies $\wof{R_1}=\wof{R_2}$, and RC graphs with the same reading word are forest equivalent. For a permutation $w$, we denote by $\mathcal{C}_\forest(w)$ the set of equivalence classes of reduced words for $w$ under $\equiv_\forest$.

The \emph{forest code} of a word $\mathbf r$ (resp.\ of an RC graph $R$) is $\code_\forest(\mathbf r) := \code(F(\overline{\mathbf r}))$ (resp.\ $\code_\forest(R):=\code(F(\overline{R}))$). It depends only on the $\equiv_\forest$-class, which we call a \emph{forest class}, and for RC graphs in $\RC_n$ we regard it as a weak composition of length $n$. We write $\mathcal{C}_\forest(a)$ for the set of forest classes of RC graphs in $\RC_n$ with forest code $a$.
\end{definition}

\begin{remark}[Compatible-sequence conventions]
The compatible-sequence convention used for slide polynomials below is the reverse of the convention used in the forest-polynomial literature: weakly increasing sequences here correspond to weakly decreasing sequences there. Since the forest invariant is applied to the reversed reading word $\mathrm{rev}(\mathrm{word}(R))$, this reversal of word order compensates for the reversal of convention. Thus the forest-equivalence classes and their weight-generating functions are expressed consistently in the conventions of this paper; in particular, no additional change of weights is needed in Theorem~\ref{theorem:forestformula}.
\end{remark}

\begin{theorem}[{\cite[Theorem~1.4]{nadeau2024forest}}] \label{theorem:forestforest}
  Let $w\in S_\infty$. Then
  $$\sch_w(x) = \sum_{C\in \mathcal{C}_\forest(w)} \forest_{\code_\forest(C)}(x).$$
\end{theorem}

\begin{theorem}[{\cite[Theorem~\ref*{B-theorem:forestformula}(I)]{samuelbialgebra}}]\label{theorem:forestformula}
Let $R\in \RC_n$ be a bounded RC graph. Then the forest polynomial $\forest_{\fcode(R)}$ is equal to
$$\forest_{\fcode(R)}(x)=\sum_{R'\equiv_{\forest} R} x^{\wtt(R')},$$
the sum over $R'\in\RC_n$.
\end{theorem}

This is a reformulation of \cite[Proposition~5.10]{nadeau2024forest}: for an equivalence class $W$ of reduced words with a common LBS labeling $P$, $\forest_{\code_\forest(W)}$ is the sum over the compatible sequences of the words in $W$, which are exactly the RC graphs in the corresponding forest class. Guo and Woodruff \cite{guo2026forest} have also published an RC graph formula for forest polynomials. The other input from \cite{nadeau2024forest} that we need is the stability of shuffles of forest classes under $\equiv_\forest$, which is the heart of Nadeau and Tewari's proof of positivity of the $\beta_{a,b}^c$; it is recalled in the precise form we use in the proof of Lemma~\ref{lemma:foreststable}.

\subsection{Slide polynomials and quasi-Yamanouchi RC graphs}\label{section:slidebackground}

\begin{definition}[{\cite{assaf2017slide}}]\label{definition:slide}
For a word $\mathbf{r}=\mathbf{r}_1\mathbf{r}_2\cdots \mathbf{r}_k$ of positive integers, a sequence of positive integers $a_1a_2\cdots a_k$ is \emph{compatible} with $\mathbf{r}$ if $a_i\leq \mathbf{r}_i$ for all $i$, $a_i\leq a_{i+1}$ for all $i$, and $a_i<a_{i+1}$ if $\mathbf{r}_i<\mathbf{r}_{i+1}$. The \emph{slide polynomial} $\mathfrak{F}(\mathbf{r})$ of a word $\mathbf{r}$ is the sum of the monomials $x^a$ over all compatible sequences $a$ of $\mathbf{r}$. For an injective word $W$ on $\overline{\mathbb{Z}}$ we set $\mathfrak F(W)=\mathfrak F(\mathrm{val}(W))$.

For a weak composition $a$, the \emph{fundamental slide polynomial} $\mathfrak{F}_a$ is $\mathfrak{F}(\mathbf r)$ for any word $\mathbf r$ having $a$ as a compatible sequence with $\mathrm{flat}(a)$ coarsest among its compatible sequences, where $\mathrm{flat}(a)$ is the composition obtained from $a$ by deleting zeros; explicitly \cite{assaf2017slide},
$$\mathfrak{F}_a = \sum_{b} x^b,$$
the sum over weak compositions $b\geq a$ in dominance order such that $\mathrm{flat}(b)$ refines $\mathrm{flat}(a)$. The $\mathfrak F_a$ with $a$ of length $n$ form a $\mathbb{Z}$-basis of $\coma_n$, and $\mathfrak F(\mathbf r)=\mathfrak F_{a}$ where $a$ is the unique compatible sequence of $\mathbf r$ with coarsest flattening.

A \emph{quasi-Yamanouchi} RC graph is an RC graph $R$ such that if $R'$ is another RC graph with $\mathrm{word}(R)=\mathrm{word}(R')$ and $\hr(R)=\hr(R')$, then $\wtt(R)\leq \wtt(R')$ in dominance order and $\mathrm{flat}(\wtt(R'))$ refines $\mathrm{flat}(\wtt(R))$. The set of RC graphs $A$ such that $\mathrm{word}(A)=\mathrm{word}(R)$ and $\hr(A)=\hr(R)$, which we call the \emph{slide class} of $R$, contains a unique quasi-Yamanouchi element, which we denote by $\QY(R)$. By the identification of RC graphs with compatible sequences, the weight generating function of the slide class of $R$ is $\mathfrak F(\mathrm{word}(R)) = \mathfrak F_{\wtt(\QY(R))}$, and we write $\mathfrak F(R):=\mathfrak F(\mathrm{word}(R))$.
\end{definition}

\begin{theorem}[{\cite[Theorem~3.13]{assaf2017slide}}]\label{theorem:assafsearles}
  For a permutation $w\in S_\infty$ with $\maxd(w)\leq n$, we have the formula
  $$\sch_w(x) = \sum_{R} \mathfrak{F}_{\wtt(R)}(x)$$
  where the sum is over all quasi-Yamanouchi RC graphs $R\in\RC_n(w)$.
\end{theorem}

Fundamental slide polynomials multiply by a shuffle rule \cite[Theorem~5.11]{assaf2017slide}, which in the language of injective words reads as follows \cite[\S5]{nadeau2024ppartitions}: for injective words $W_1,W_2$ on $\overline{\mathbb Z}$ with disjoint letter sets,
\begin{equation}\label{eq:slideshuffle}
\mathfrak F(W_1)\,\mathfrak F(W_2) = \sum_{X\in\mathsf{Sh}(W_1,W_2)}\mathfrak F(X).
\end{equation}
The \emph{slide class} of an injective word $W$ is the connected component of $W$ in the quasi-crystal of injective words of Cain--Guilherme--Malheiro \cite[\S5]{cain2023quasicrystals}, recalled below, whose character is $\mathfrak F(W)$.

Finally, we recall the quasi-crystal framework of Cain, Guilherme, and Malheiro \cite{cain2023quasicrystals} to the extent that we use it. A \emph{quasi-crystal} of type $A_\infty$ is a set $Q$ with a weight map $\wtt:Q\to\mathbb{Z}^\infty$, partial operators $\ddot e_i,\ddot f_i$ ($i\geq 1$) and maps $\ddot\varepsilon_i,\ddot\varphi_i:Q\to\mathbb{Z}\cup\{\pm\infty\}$ satisfying the axioms of \cite[Definition~3.1]{cain2023quasicrystals}; a quasi-crystal is \emph{seminormal} \cite[Definition~3.9]{cain2023quasicrystals} if $\ddot\varepsilon_i,\ddot\varphi_i$ take values in $\mathbb{Z}_{\geq 0}\cup\{+\infty\}$, with $\ddot\varepsilon_i(x)=\max\{k:\ddot e_i^k(x)\text{ defined}\}$ when finite and similarly for $\ddot\varphi_i$, and $\ddot\varepsilon_i(x)=+\infty$ exactly when both $\ddot e_i(x)$ and $\ddot f_i(x)$ are undefined. The difference from a Kashiwara crystal is that an element may have $\ddot\varepsilon_i=\ddot\varphi_i=+\infty$, in which case it is \emph{$i$-inert}: both operators are undefined, and this is not a boundary of a string. The \emph{quasi-tensor product} $Q_1\,\ddot\otimes\,Q_2$ of seminormal quasi-crystals \cite[\S5]{cain2023quasicrystals} is the Kashiwara tensor product \eqref{eq:tensorrule}, except that whenever $\ddot\varphi_i(x_1)>0$ and $\ddot\varepsilon_i(x_2)>0$ the element $x_1\,\ddot\otimes\,x_2$ is declared $i$-inert. The quasi-tensor product of seminormal quasi-crystals is seminormal and $\ddot\otimes$ is associative \cite[Theorem~5.1]{cain2023quasicrystals}, and every seminormal quasi-crystal decomposes uniquely as a disjoint union of its connected components \cite[Proposition~4.6]{cain2023quasicrystals}. The quasi-crystal of words on a totally ordered alphabet is the iterated quasi-tensor product of the standard crystal of letters \cite[\S5]{cain2023quasicrystals}; applied to the alphabet $\overline{\mathbb Z}$ with $\mathrm{val}$ as the weight, this is the quasi-crystal of injective words referred to above.

\clearpage

\section{The Littlewood-Richardson rule for forest polynomials}\label{section:forestrule}

In this section we build up the necessary theory to state and prove the Littlewood-Richardson rule for forest polynomials, Theorem~\ref{theorem:forestpolyLR}, which expresses the coefficient of an individual forest polynomial in the expansion of a product of two forest polynomials as the cardinality of an explicitly defined set. Explicitly, it is the characterization of the coefficient $\beta_{ab}^c$ in the equation
$$\forest_a\forest_b = \sum_{c}\beta_{ab}^c\forest_c$$
enumeratively.

The statement of Theorem~\ref{theorem:forestpolyLR} concerns forest polynomials directly rather than their duals. Its proof draws on the row-decomposition machinery developed in the companion paper \cite{samuelbialgebra} in preparation for the dual LR rules proved there, but everything that is needed has been collected in \S\ref{section:background}, and a reader interested only in this result need not work through the companion paper first.

That the coefficients $\beta_{ab}^c$ are nonnegative integers is not new. Bergeron, Gagnon, Nadeau, Spink, and Tewari \cite{bergeron2025equivariant} prove constructively that the forest polynomials multiply with nonnegative integer (in fact, in the equivariant setting, Graham-positive) structure constants, by extracting them from a positive straightening algorithm: an iterative rewriting procedure in an augmented version of the positive Thompson monoid that produces, given $a$ and $b$, a positive expansion of $\forest_a \forest_b$. The authors of \cite{bergeron2025equivariant} are explicit, however, that this is not an enumerative rule in the spirit of Schubert calculus: the output of their algorithm depends on the rewriting choices made along the way, and a given $\beta_{ab}^c$ is recovered only as the cumulative sum of contributions accumulated through the procedure, rather than as the cardinality of any prescribed combinatorial set.

\subsection*{Outline of the proof}

The proof proceeds in four steps.

In \S\ref{section:multiplactic} we define an associative product $\boxtimes$ on $\BRC_n$, the \emph{squash product}, built directly out of the row-decomposition operations of \S\ref{section:cliptrim} (a horizontal concatenation of two RC graphs of common height, followed by $\clp$ to recover an $n$-row representative). The product behaves quite poorly as a global product on $\BRC$. However, on full Grassmannian RC graphs of height $n$, $\boxtimes$ turns out fortuitously to coincide with the plactic monoid product on semistandard Young tableaux of shape with parts at most $n$, modulo the customary transposition between RC-graph and tableau conventions (Theorem~\ref{theorem:placticiso}). Furthermore, it models the standard theory of the polynomial ring in $n$ variables as a module over the ring of symmetric polynomials with basis the Schubert polynomials indexed by permutations in $S_n$. This leverage (a fairly well-behaved right module action of the ring of Grassmannian RC graphs of a fixed height) is the key to expressing Schubert multiplication in terms of $\boxtimes$. 

The squash product induces a canonical projection $\RC_n:\mathcal{M}_n\to\BRC_n$, where $\mathcal{M}_n$ is the \emph{multiplactic algebra} obtained by tensoring the height-stratified plactic algebras together.

The multiplactic algebra carries enough structure to express Schubert multiplication. With a specific choice of sum of elementary tensors corresponding to elementary symmetric factorizations, $\mathcal{M}_n$ supports a signed expansion whose image under $\RC_n$ recovers the Schubert structure constants $c_{u,v}^w$ correctly (Theorem~\ref{theorem:schubproduct}). That is, for permutations $u$ and $v$, there are elements of $\mathcal{M}_n$ whose images under $\RC_n$ are precisely the sums of RC graphs in $\RC_n(u)$ and $\RC_n(v)$ respectively, with apparently extraneous signed terms in $\mathcal{M}_n$ that cancel out under $\RC_n$, whose product yields the correct RC graphs with the correct multiplicities for the Schubert structure constants $c_{u,v}^w$ upon projection of the product in $\mathcal{M}_n$ to $\BRC_n$. Those familiar with quantum Schubert polynomials and the Fomin-Kirillov algebra may be less surprised by this apparent miracle than those who are not.

This signed model is genuinely \emph{cancellative}: distinct terms in the signed expansion in $\mathcal{M}_n$ that project under $\RC_n$ to the same RC graph with opposite signs, completely canceling, unfortunately can contribute nontrivial terms to the product. That is to say, simply removing these terms causes the correctness to fail, and it is not simply an artifact of presentation.

For forest polynomials, however, the obstruction is weaker. Removing the extraneous terms when we restrict to forest RC graphs in order to model multiplication of forest polynomials squeaks by well enough to provide an LR rule, despite not being consistent enough to descend to a ring structure on labeled indexed forests.

\subsection{The multiplactic algebra}\label{section:multiplactic}
\begin{definition}
Let $R_1$ and $R_2$ be RC graphs with $\hr(R_1)=\hr(R_2)$. Let $N$ be the position of the last unfixed point of $\wof{R_1}$. We define
$$R_1\sqcup R_2 = R_1\cup \{(i,j+N)\mid (i,j)\in R_2\}$$
This is an RC graph with $|R_1\sqcup R_2| = |R_1| + |R_2|$ and $\hr(R_1\sqcup R_2)=\hr(R_1) + N$, though at most only the first $\hr(R_1)$ rows are occupied.  It is not too difficult to see that
$$\wof{R_1\sqcup R_2} = \wof{R_1}\shup^{N}\wof{R_2}$$

We then define the \emph{squash product} of $R_1$ and $R_2$ to be
$$R_1\boxtimes R_2 = \clp^{\hr(R_1)}(R_1\sqcup R_2)$$

\end{definition}

\begin{proposition} \label{proposition:boxproduct}
  Let $u\in S_\infty$ be a permutation such that $\maxd(u)\leq n$ and let $v$ be an $n$-Grassmannian permutation. Then we have that
  $$\boxtimes:\RC_n(u)\otimes \RC_n(v)\to \bigoplus_{w\in S_\infty} \RC_n(w)^{\oplus c_{uv}^w}$$
  is an isomorphism of crystal graphs. 
\end{proposition}
\begin{proof}
Clearly the map is weight preserving. We defer the assertion that the image is precisely correct by allowing us to expand the codomain to the set $\RC_n$. 
Let us show that the function commutes with the crystal operators. Let $R_1\in \RC_n(u)$ and $R_2\in \RC_n(v)$, and let $R=R_1\sqcup R_2$. Consider $f_i(R_1\otimes R_2)$. If $\varphi_i(R_1)>\varepsilon_i(R_2)$, then the unpaired elements in row $i$ in $R_1$ exhaust the unpaired elements in row $i+1$ in $R_2$. Therefore, only elements of $R_1$ are affected by $f_i$, as in the definition of the tensor product. If instead $\varphi_i(R_1)\leq \varepsilon_i(R_2)$, then all unpaired elements of $R_1$ in row $i$ are paired with elements of $R_2$ in row $i+1$, and so only elements of $R_2$ are affected by $f_i$. The same reasoning applies to $e_i$. The fact that the isomorphism induces the specific decomposition follows from \cite{kouno2020decomposition}.
\end{proof}

\begin{definition}
We define an RC graph $R$ of $\hr(R)=k$ to be \emph{full Grassmannian} if $\wof{R}$ is a $k$-Grassmannian permutation (that is, it has at most one descent, which would be at position $k$). We define $\grass_n$ to be the $\mathbb{Z}$-submodule of $\BRC_n$ generated by all full Grassmannian RC graphs with $\hr(R)= n$.

We define $\plac_n$ to be the plactic algebra on $[n]$, the free abelian group on the set of words in $[n]$ modulo the Knuth relations. We take as the basis of $\plac_n$ the set of row reading words of semistandard Young tableaux \emph{in French notation} (upside-down).

Recall that the Young/Ferrers diagram of a partition (in French notation) is a set of ``boxes'' on a grid, with $\lambda_i$ boxes in row $i$ for the first $\ell(\lambda)$ rows, where the rows are indexed from the bottom up. The notation $\mathrm{SSYT}_n$ denotes the set of semistandard Young tableaux, which are labelings of the boxes of such diagrams with elements of $[n]$ such that the numbers strictly increase along the columns and weakly increase along the rows. The shape of a semistandard Young tableau is the partition corresponding to the underlying Young diagram, and the weight is the weak composition whose $i$-th part counts the number of boxes labeled $i$.

Consider the set $I(v)$ of right inversions of an $n$-Grassmannian permutation $v$, that is to say the pairs $(i,j)$ with $i<j$ and $v(i)>v(j)$. An RC graph $R\in\RC_n(v)$ corresponds bijectively to a labeling $T_R:I(v)\to [n]$ by defining
$$T_R(\rtt(i,j)) = i$$
for an element $(i,j)\in R$ (see, for example, \cite{axelrodfreed2025inversionstableaux}). If $R_0$ is the bottom RC graph of $v$, then there is a bijection between the elements of $I(v)$ and the boxes of the Young diagram of $v$ given by
$$(i,j) \mapsto (i + \ell(\mathrm{shape}(v)) - n, j)$$
Hence there is a linear isomorphism
$$P_n:\mathbb{Z}^{\oplus\mathrm{SSYT}_n}\to \grass_n$$
sending $T$ to the RC graph $R$ such that $T_R = T$.
\end{definition}

\begin{lemma}
Suppose $R$ is full $n$-Grassmannian such that all rows after row $r$ for some $r\leq n$ are empty. Then $\clp^r(R)$ is full $r$-Grassmannian, and if $T\in \SSYT(n)$ satisfies $P_n(T)=R$, then $P_r(T)=\clp^r(R)$.
\end{lemma}
\begin{proof}
From Theorem \ref{theorem:zero_bijection} and the fact that Schur polynomials are stable, we have that $\clp^r(R)$ is full $r$-Grassmannian of the same shape and weight. Invoking the more specific Proposition \ref{proposition:zerocrystal} and theory of Kashiwara crystals, we have that $P_r(T)=\clp^r(R)$.
\end{proof}

\begin{theorem} \label{theorem:placticiso}
We have that the bilinear map $\boxtimes: \grass_n\times \grass_n\to\grass_n$ is an associative product such that $P_n:\plac_n\to \grass_n$ is an isomorphism of rings.
\end{theorem}
\begin{proof}
Consider the function 
$$\tau_n:\plac_n\otimes \plac_n\to \bigoplus_{N = n}^\infty \BRC_N$$ 
defined on basis elements by
$$\tau_n(T_1\otimes T_2) = P_n(T_1)\sqcup P_n(T_2)$$
Proposition \ref{proposition:boxproduct} shows that
$$\clp^n(\tau_n(T_1\otimes T_2)) = P_n(T_1)\boxtimes P_n(T_2) = P_n(T_1T_2)$$
due to the rigidity of crystal graph isomorphisms. Since $P_n$ is therefore a homomorphism of (potentially nonassociative) algebras as well as bijective, it is an isomorphism, from which it follows that $\boxtimes$ itself is associative.
\end{proof}

\begin{theorem} \label{theorem:schubmoduleiso}
  The module $\BRC_n$ is a right $\grass_n$-module under the product $\boxtimes$. Considering $\BRC_n$ as a right module over $\Lambda_n$ via the injective homomorphism $\Lambda_n\hookrightarrow \plac_n\simeq \grass_n$, the $\Lambda_n$-submodule generated by $\mathcal{S}_w(n)$ for $w\in S_n$ is isomorphic to the polynomial ring $\mathbb{Z}[x_1,\ldots,x_n]$ as a $\Lambda_n$-module via the homomorphism $\mathcal{S}_w(n)\mapsto\sch_w(x)$, and in particular is a free right $\Lambda_n$-module with basis $\{\mathcal{S}_w(n)\mid w\in S_n\}$.
\end{theorem}
\begin{proof}
The fact that $\BRC_n$ is a right $\grass_n$ module follows from Proposition \ref{proposition:boxproduct} and the statement that the tensor product of crystal graphs is associative. The injective homomorphism of rings $\Lambda_n\hookrightarrow \plac_n\simeq \grass_n$ is given by sending a Schur polynomial $s_\lambda$ to the sum of all words of semistandard tableaux of shape $\lambda$ is realized through $P_n$. The map $\BRC_n\to \mathbb{Z}[x_1,\ldots,x_n]$ sending $R$ to $x^{\wtt(R)}$ is clearly a surjective homomorphism of abelian groups. Restricting to the submodule $\mathrm{Schub}_n$, since $\mathcal{S}_w(n)\mapsto \sch_w^n(x)$ for all $w$ with $\maxd(w)\leq n$, we obtain an isomorphism of $\mathbb{Z}$-modules at the very least. Proposition \ref{proposition:boxproduct} ensures that the restriction to this submodule is in fact a $\Lambda_n$-isomorphism. Thus, the $\Lambda_n$-submodule generated by the $\mathcal{S}_w(n)$ is free with basis $\{\mathcal{S}_w(n)\mid w\in S_n\}$.
\end{proof}

\begin{definition}[The multiplactic algebra]
Consider the ring $\mathcal{M}_n$ defined as
$$\mathcal{M}_n = \bigotimes_{k=1}^n \grass_k$$

We then define a linear homomorphism
$$\mathcal{M}_n\to \BRC_n$$
by
$$\RC(A_1\otimes A_2\otimes \cdots \otimes A_k) = A_1\boxtimes A_2\boxtimes \cdots \boxtimes A_k$$
where we increase the size of the left factor to match the right factor as needed, finally resizing the end result to height $n$.

We may avoid the clunky tensor product notation by noting that we may identify which tensor factor an element of $\mathcal{M}_n$ belongs to by its height, and so we may write elements of $\mathcal{M}_n$ as formal products of full-Grassmannian RC graphs, with the understanding that the $k$-th factor of the product is the submodule $\grass_k$.

For each $1\leq k\leq n$ and binary vector $\beta\in\{0,1\}^k$, define $\mathbf{E}^{k;\beta}$ to be the unique $k$-Grassmannian RC graph for a partition with a single column satisfying
$$\wtt_j(\mathbf{E}^{k;\beta}) = \beta_j$$
for $1\leq j\leq k$. Equivalently, the rows of $\mathbf{E}^{k;\beta}$ that carry a crossing are exactly those indexed by the support of $\beta$, namely $\{j\in\{1,\ldots,k\}:\beta_j=1\}$, and the degree (column length) of $\mathbf{E}^{k;\beta}$ is $|\beta|:=\beta_1+\cdots+\beta_k$.

Define $\mathbf{E}_{p}^k$ by
$$\mathbf{E}_{p}^k = \sum_{\substack{\beta\in\{0,1\}^k\\ |\beta|=p}} \mathbf{E}^{k;\beta}$$
where the sum is over all binary vectors of length $k$ with exactly $p$ ones, in bijection with the $\binom{k}{p}$ size-$p$ subsets of $\{1,\ldots,k\}$.
\end{definition}

\begin{proposition} \label{proposition:eproductworks}
For any $p_1,p_2,\ldots,p_m$ and $k_1\leq k_2\leq\ldots\leq k_m\leq n$ with $p_i\leq k_i$ for all $i$, we have that
$$\RC_n(\mathbf{E}_{p_1}^{k_1}\mathbf{E}_{p_2}^{k_2}\cdots\mathbf{E}_{p_m}^{k_m}) = \sum_{w} c_w\mathcal{S}_w(n)$$
where $c_w\in \mathbb{Z}_{\geq 0}$ are the coefficients such that
$$e_{p_1,k_1}^n\cdots e_{p_m,k_m}^n = \sum_w c_w \sch_w^n(x)\in \coma_n$$
\end{proposition}
\begin{proof}
We prove this by induction on $m$. For an elementary tensor $A_1\otimes \cdots \otimes A_m$, the result follows directly from the definition of $\RC_n$. Assuming the inductive hypothesis that this formula holds for tensors of length less than $m$, then for a tensor of length $m$ we have that
$$\RC_n(\mathbf{E}_{p_1}^{k_1}\mathbf{E}_{p_2}^{k_2}\cdots\mathbf{E}_{p_m}^{k_m}) = \RC_n(\mathbf{E}_{p_1}^{k_1}\mathbf{E}_{p_2}^{k_2}\cdots\mathbf{E}_{p_{m-1}}^{k_{m-1}})\boxtimes \RC_n(\mathbf{E}_{p_m}^{k_m})=\sum_{w'} c_{w'}' \mathcal{S}_{w'}(n)\boxtimes \mathbf{E}_{p_m}^{k_m}$$
by Theorem \ref{theorem:schubmoduleiso}, since $\mathbf{E}_{p_m}^{k_m}\in\grass_{k_m}$, and we are ensured by the increasing order of $k_i$ that the left factor can be clipped to height $k_m$ without the permutation changing. We then apply Proposition \ref{proposition:boxproduct} to obtain the result for a single term of this type by induction. Invoking the fact that the image lies in $\mathrm{Schub}_n$, the result follows by applying Theorem \ref{theorem:schubmoduleiso}.
\end{proof}

\begin{definition}
  Suppose $\maxd(w)\leq n$. We define an element $\mathbf{S}_w(n)\in \mathcal{M}_n$ by
  $$\mathbf{S}_w(n) = \sum_{\alpha} E_w^{\alpha;n}\mathbf{E}_{\alpha_1}^1 \mathbf{E}_{\alpha_2}^2 \cdots  \mathbf{E}_{\alpha_n}^n\mathbf{E}_{\alpha_{n+1}}^n\cdots \mathbf{E}_{\alpha_N}^n$$
  where we recall that $E_w^{\alpha;n}$ represents the coefficient of the elementary basis element of $\coma_n$ indexed by the weak composition $\alpha$ in the expansion of $\sch_w^n$. We observe by Proposition \ref{proposition:eproductworks} that $\RC_n(\mathbf{S}_w(n)) = \mathcal{S}_w(n)$, so that $\mathbf{S}_w(n)$ is a lift of $\mathcal{S}_w(n)$ to $\mathcal{M}_n$, which usually has signs (\emph{but not always}; see, for example, \cite{woodruff2025single} for a classification of permutations for which $\mathbf{S}_w(n)$ has a single SEM term).
\end{definition}

\begin{theorem} \label{theorem:schubproduct}
Let $u,v\in S_\infty$ satisfy $\maxd(u),\maxd(v)\leq n$. Then
$$\RC_n(\mathbf{S}_u(n)\mathbf{S}_v(n)) = \sum_{w} c_{uv}^w \mathcal{S}_w(n)$$
where $c_{uv}^w$ are the Schubert structure constants.
\end{theorem}
\begin{proof}
  Explicitly, write
  $$\mathbf{S}_u(n)\mathbf{S}_v(n) = \sum_{\alpha,\beta} E_u^{\alpha;n}E_v^{\beta;n} \mathbf{E}_{\alpha_1}^1 \cdots  \mathbf{E}_{\alpha_N}^n \mathbf{E}_{\beta_1}^1 \cdots  \mathbf{E}_{\beta_N}^n$$
  The relations of the ring allow us to slide smaller terms in $\mathbf{S}_v(n)$ past strictly larger terms in $\mathbf{S}_u(n)$, so we may rearrange the product in increasing order of height. The result then follows from Proposition \ref{proposition:eproductworks} and the fact that the image of the element under $\RC_n$ lies in $\mathrm{Schub}_n$, so that we may invoke Theorem \ref{theorem:schubmoduleiso}.
\end{proof}

\begin{remark}
It may be unsurprising to the reader at this point that the kernel of $\RC_n:\mathcal{M}_n\to \BRC_n$ \emph{is not an ideal}. If it were, then the product on $\mathcal{M}_n$ would descend to a term by term product such that
$$\mathcal{S}_u(n)\cdot \mathcal{S}_v(n) = \sum_{w} c_{uv}^w \mathcal{S}_w(n)$$
and the problem of finding a positive formula for the Schubert structure constants would be solved, due to the fact that $\mathcal{S}_u(n)$ and $\mathcal{S}_v(n)$ have completely nonnegative coefficients. Unfortunately, this is simply not true. The negative terms in $\mathbf{S}_u(n)$ are essential in making the product work.
\end{remark}

\subsection{The lift product}\label{section:liftproduct}

We now define a canonical representative of an RC graph in $\mathcal{M}_n$, an elementary tensor without multiplicity, that allows us to project back to define an associative product, but as necessitated by the remark above this does not correctly model Schubert polynomial multiplication. However, we will see that it is correct enough to give a Littlewood-Richardson rule for forest polynomials.

\begin{proposition} \label{proposition:elemfactor}
  Let $u\in S_\infty$ be a permutation and suppose $R\in \RC_n(u)$. Let $\lambda$ be a partition and let $w$ be the corresponding dominant permutation with $\code^*(w)=\lambda$. Assume that $\ell(uw^{-1}) = \ell(w)-\ell(u)$. Let $c$ be a weak composition such that $c_i\leq \lambda_{N + 1 - i}$ for all $i$, where $N=\ell(\lambda)$. Then the number of sequences of binary vectors $\beta_1,\ldots,\beta_N$ such that $\beta_i\in\{0,1\}^{\lambda_{N + 1 - i}}$ with $|\beta_i|=c_i$ satisfying
  $$\mathbf{E}^{\lambda_N;\beta_1}\boxtimes\cdots\boxtimes \mathbf{E}^{\lambda_1;\beta_N} = R$$
is equal to the coefficient of $x^c$ in $\sch_{uw^{-1}}(x)$.
\end{proposition}
\begin{proof}
  Recall the Cauchy formula for $\sch_w(x;-y)$:
  $$\sch_w(x;-y)=\sum_u \sch_u(x)\sch_{uw^{-1}}(y)$$
  Also,
  $$\sch_w(x;-y) = E_{\lambda_1,\lambda_1}(x;-y_1)\cdots E_{\lambda_N,\lambda_N}(x;-y_N)$$

Replace the factorial elementary symmetric polynomial $E_{p,p}(x;-y_i)$ for each $p$ with
$$\mathbf{E}_p^p(-y_i)\in \mathcal{M}_n[y_1,\ldots,y_N],$$
defined by
$$\mathbf{E}_p^p(-y_i) = \sum_{j=0}^p y_i^j \mathbf{E}_{p-j}^p.$$ 
The result follows from Proposition \ref{proposition:eproductworks}.
\end{proof}

\begin{definition}
  Define a function $\mathrm{lift}_n:\RC_n\to \mathcal{M}_n$ as follows. For an RC graph $R\in \RC_n$, let $\lambda$ be a partition of the form 
$$(\overbrace{n,\ldots,n}^{p},n-1,n-2,\ldots,1)$$
where $p$ is sufficiently large such that $\ell(\wof{R}\dom_\lambda^{-1}) = \ell(\dom_\lambda) - \ell(\wof{R})$.  Let $c$ be the code of the permutation $\wof{R}\dom_\lambda^{-1}$. Then we define
$$\mathrm{lift}_n(R) = E_1 \cdots  E_{\ell(c)}$$
where $|E_i| = \lambda_{N + 1 - i} - c_{N + 1 - i}$ and $\hr(E_i) = \lambda_{N+1-i}$ for all $i$, with $E_i$ all elementary symmetric full Grassmannian RC graphs, which is guaranteed to exist and be unique by Proposition \ref{proposition:elemfactor}. 
\end{definition}

Some justification that $\mathrm{lift}_n$ is well-defined is in order. For each choice of $\lambda$, the existence and uniqueness of the factorization follows from Proposition \ref{proposition:elemfactor}. However, $\lambda$ depends on a choice of $p$ ``sufficiently large.'' We must show that the factorization is independent of this choice. If we increase $p$ by one, then $\lambda$ is replaced by $\lambda' = (\overbrace{n,\ldots,n}^{p+1},n-1,n-2,\ldots,1)$. The code of $\wof{R}(\dom_{\lambda'})^{-1}$ is then obtained from the code of $\wof{R}(\dom_\lambda)^{-1}$ by prepending $n$, so the same factorization works for both choices of $p$.

\begin{definition}
We define a product $*$ on $\BRC_n$ by
$$R_1*R_2 = \RC_n(\mathrm{lift}_n(R_1)\mathrm{lift}_n(R_2))$$
\end{definition}

\begin{lemma}[Weight-additivity of the lift product]\label{lemma:liftweight}
For any $R_1, R_2\in\BRC_n$,
$$\wtt(R_1 * R_2) = \wtt(R_1) + \wtt(R_2),$$
where addition of weight vectors is componentwise. Equivalently, the weight-evaluation map
$$\mathrm{ev}:\BRC_n\to\mathbb{Z}[x_1,\ldots,x_n],\qquad \mathrm{ev}(R) = x^{\wtt(R)},$$
is a homomorphism of rings from $(\BRC_n,*)$ to the polynomial ring.
\end{lemma}
\begin{proof}
We first observe that the squash product preserves row weights: by Definition of $\sqcup$, the columns of $R_2$ are shifted strictly past those of $R_1$ but their row indices are unchanged, so $\wtt(R_1\sqcup R_2)_i = \wtt(R_1)_i + \wtt(R_2)_i$ for $i\leq \hr(R_1)$ and $0$ otherwise. Since the rows beyond $\hr(R_1)$ of $R_1\sqcup R_2$ are empty, each iterate of $\zeromap$ used to form $\clp^{\hr(R_1)}(R_1\sqcup R_2)$ acts on an empty last row and merely truncates it; in particular the row weights for rows $1,\ldots,\hr(R_1)$ are preserved. Hence
$$\wtt(R_1\boxtimes R_2)_i = \wtt(R_1)_i + \wtt(R_2)_i\qquad (1\leq i\leq \hr(R_1))$$
whenever $R_1$ and $R_2$ have the same height. By induction, the same identity holds for any $\boxtimes$-product of RC graphs of common height: row weights add componentwise.

Now apply this to elementary multiplactic factors. For $\beta\in\{0,1\}^k$, we have $\wtt_j(\mathbf{E}^{k;\beta}) = \beta_j$ by definition. If $\mathrm{lift}_n(R) = E_1\cdots E_m$ with $E_i = \mathbf{E}^{k_i;\beta_i}$ a full-Grassmannian elementary RC graph, then $R = \RC_n(E_1\otimes\cdots\otimes E_m)$ is the iterated $\boxtimes$-product of the $E_i$.

Padding each $E_i$ by empty rows up to common height $n$, the row-weight-additivity of $\boxtimes$ established above gives
$$\wtt_j(R) = \sum_{i:\,k_i\geq j} (\beta_i)_j\qquad (1\leq j\leq n).$$
Applying the same identity to $R_1*R_2 = \RC_n(\mathrm{lift}_n(R_1)\,\mathrm{lift}_n(R_2))$, whose multiplactic factorization is the concatenation of those of $R_1$ and $R_2$, yields
$$\wtt(R_1*R_2) = \wtt(R_1) + \wtt(R_2)$$
componentwise. The ring-homomorphism statement is immediate:
$$\mathrm{ev}(R_1*R_2) = x^{\wtt(R_1)+\wtt(R_2)} = x^{\wtt(R_1)}\,x^{\wtt(R_2)} = \mathrm{ev}(R_1)\mathrm{ev}(R_2).$$
\end{proof}

\begin{lemma}
  Let $R$ be an RC graph and suppose
$$\mathrm{lift}_n(R) = A_1\cdots A_n$$
with $\hr(A_i) = i$ for all $i$. Let $1\leq j\leq n$ and let 
$$R^{\mathrm{left}} = \RC_n(A_1\cdots  A_j)$$
and
$$R^{\mathrm{right}} = \RC_n(A_{j+1}\cdots  A_n)$$
Then
$$\mathrm{lift}_n(R^{\mathrm{left}}) = A_1\cdots  A_j$$
and
$$\mathrm{lift}_n(R^{\mathrm{right}}) = A_{j+1}\cdots  A_n$$
\end{lemma}
\begin{proof}
This follows from the fact that if a partition $\lambda$ is split into an initial section $\lambda_1$ of length $n-j$ and a final section $\lambda_2$ of length $j$,  if we define $c_1$ and $c_2$ to be the codes of $\wof{R^{\mathrm{left}}}\dom_{\lambda_1}^{-1}$ and $\wof{R^{\mathrm{right}}}\dom_{\lambda_2}^{-1}$, then the code of $\wof{R}(\dom_{\lambda_1}\dom_{\lambda_2})^{-1}$ is obtained by concatenating $c_1$ and $c_2$. The result then follows from Proposition \ref{proposition:elemfactor}. Note that $\dom_{\lambda_1}\dom_{\lambda_2}$ is the dominant permutation with code $\lambda_1\cup \lambda_2$, where the union is the multiset union considering these as partitions.
\end{proof}

\begin{lemma}[Lift factors at disjoint height ranges commute]\label{lemma:liftcommute}
Let $A\in\mathcal{M}_n$ be a tensor of elementary RC factors all of whose heights are at most $k$, and let $B\in\mathcal{M}_n$ be a tensor of elementary RC factors all of whose heights are strictly greater than $k$. Then
$$A\cdot B = B\cdot A\qquad\text{in }\mathcal{M}_n,$$
where the product is the lift product $*$ on $\BRC_n$ transferred to $\mathcal{M}_n$ via $\mathrm{lift}_n$. Equivalently, given any RC graph $R$ admitting a lift-factorization $R=R^{\mathrm{left}}*R^{\mathrm{right}}$ with $\hr(R^{\mathrm{left}})\leq k<\hr(R^{\mathrm{right}})$, the factors $R^{\mathrm{left}}$ and $R^{\mathrm{right}}$ commute under $*$ with any analogous factors arising from other RC graphs.
\end{lemma}
\begin{proof}
Elementary symmetric RC graphs of different heights commute trivially, and the lift map turns this into commutativity of the corresponding $\boxtimes$-factors in $\mathcal{M}_n$.
\end{proof}

\begin{theorem}
The product $*:\BRC_n\times\BRC_n\to\BRC_n$ is associative and is generated by the elementary symmetric full-Grassmannian RC graphs of height at most $n$.
\end{theorem}
\begin{proof}
We want to prove that for RC graphs $R_1$, $R_2$, and $R_3$,
$$R_1*(R_2*R_3) = (R_1*R_2)*R_3$$
Assume first that $R_3$ is elementary symmetric of length $k$. We may uniquely factorize 
$$R_1 = R_1^{\mathrm{left}}* R_1^{\mathrm{right}}$$
$$R_2 = R_2^{\mathrm{left}}* R_2^{\mathrm{right}}$$
where $\hr(R_1^{\mathrm{left}}) \leq k$ and $\hr(R_1^{\mathrm{right}}) > k$, and similarly for $R_2^{\mathrm{left}}$ and $R_2^{\mathrm{right}}$. 
By Lemma \ref{lemma:liftcommute}, the ``left'' factors commute with the ``right'' factors, so we have that
\begin{align*}
R_1*(R_2*R_3)
&= (R_1^{\mathrm{left}}*R_1^{\mathrm{right}})*\big((R_2^{\mathrm{left}}*R_2^{\mathrm{right}})*R_3\big)\\
&= (R_1^{\mathrm{left}}*R_2^{\mathrm{left}})*R_3*(R_1^{\mathrm{right}}*R_2^{\mathrm{right}}),
\end{align*}
and similarly
\begin{align*}
(R_1*R_2)*R_3
&= \big((R_1^{\mathrm{left}}*R_1^{\mathrm{right}})*(R_2^{\mathrm{left}}*R_2^{\mathrm{right}})\big)*R_3\\
&= (R_1^{\mathrm{left}}*R_2^{\mathrm{left}})*R_3*(R_1^{\mathrm{right}}*R_2^{\mathrm{right}}).
\end{align*}
Thus both bracketings are consistently written as
$$
(R_1^{\mathrm{left}}*R_2^{\mathrm{left}})*R_3*(R_1^{\mathrm{right}}*R_2^{\mathrm{right}}),
$$
so they are equal. The key point is that the corresponding lift factors commute, so this regrouping is canonical.

An arbitrary RC graph is itself a product of elementary symmetric RC graphs of weakly increasing heights according to Proposition \ref{proposition:elemfactor}, so an induction argument proves the result for arbitrary $R_3$.
\end{proof}

\subsection{The slide quotient}\label{section:slidequotient}

\begin{proposition}\label{proposition:liftcrystaliso}
Let $A = A_1\otimes\cdots \otimes A_m$ be an elementary tensor of full Grassmannian elementary symmetric polynomial RC graphs in weakly increasing order of height, considered as an element of the tensor product of Demazure crystals. Then 
$$A\mapsto\RC_n(A)$$
extends to an isomorphism of Demazure crystals from the connected component of $A$ in the tensor product to the connected component of $\RC_n(A)\subseteq \RC_n(\wof{A})$.
\end{proposition}
\begin{proof}
This is seen to be true by invoking Proposition \ref{proposition:boxproduct}, associativity of the tensor product of crystals, and the fact that the Demazure crystal generated by $A$ is isomorphic to the connected component of $A$ in the case of complete Kashiwara crystals of left to right increasing lengths \cite{kouno2020decomposition}.
\end{proof}

\begin{definition}[Quasi-crystal structure on $\RC_n(w)$]\label{definition:quasicrystal}
Define operators $\ddot{e}_i$ and $\ddot{f}_i$ on $\RC_n(w)$ by
$$\ddot{e}_i(R) = \begin{cases}
  e_i(R) & \text{ if }\mathrm{word}(R) = \mathrm{word}(e_i(R))\\
  \perp & \text{ otherwise}
\end{cases}$$
$$\ddot{f}_i(R) = \begin{cases}
  f_i(R) & \text{ if }\mathrm{word}(R) = \mathrm{word}(f_i(R))\\
  \perp & \text{ otherwise}
\end{cases}$$
where we interpret $\mathrm{word}(\perp)$ as a formal symbol distinct from every reduced word, so that the word-equality condition fails whenever $e_i(R)$ (resp.\ $f_i(R)$) is undefined in the underlying crystal. Define the weight map $\wtt:R\mapsto\wtt(R)\in\mathbb{Z}^n$ (the weight lattice of type $A_{n-1}$, with simple roots $\alpha_i=\mathbf{e}_i-\mathbf{e}_{i+1}$), and set
\begin{align*}
\ddot{\varepsilon}_i(R) &= \begin{cases}\max\{k\geq 0:\ddot{e}_i^k(R)\neq\perp\}&\text{if }\ddot{e}_i(R)\neq\perp\text{ or }\ddot{f}_i(R)\neq\perp,\\ +\infty &\text{if }\ddot{e}_i(R)=\ddot{f}_i(R)=\perp,\end{cases}\\
\ddot{\varphi}_i(R) &= \ddot{\varepsilon}_i(R) + \wtt_i(R)-\wtt_{i+1}(R),
\end{align*}
with the convention $m+(+\infty)=+\infty$ for $m\in\mathbb{Z}$.

These data equip $\RC_n(w)$ with a seminormal quasi-crystal structure (verified in Proposition \ref{proposition:quasicrystalaxioms}). Following \cite{cain2023quasicrystals}, we define the quasi-tensor product $A\,\ddot{\otimes}\,B$ of two quasi-crystals as the ordinary crystal tensor product, except that whenever $i$ is an index such that $\ddot{\varphi}_i(A)>0$ and $\ddot{\varepsilon}_i(B)>0$, we set $\ddot{e}_i(A\,\ddot{\otimes}\,B)=\ddot{f}_i(A\,\ddot{\otimes}\,B)=\perp$ (undefined).
\end{definition}

\begin{proposition}\label{proposition:quasicrystalaxioms}
The structure $(\RC_n(w),\,\ddot{e}_i,\,\ddot{f}_i,\,\wtt,\,\ddot{\varepsilon}_i,\,\ddot{\varphi}_i)_{i\geq 1}$ defined above is a seminormal quasi-crystal of type $A_{\infty}$ in the sense of \cite[Definitions~3.1,~3.9]{cain2023quasicrystals}.
\end{proposition}
\begin{proof}
We verify the six axioms of \cite[Definition~3.1]{cain2023quasicrystals} in turn. Throughout, $R\in\RC_n(w)$ and $i\geq 1$.

\smallskip\noindent\emph{Axiom (4) (inverse property).}
Suppose $\ddot{e}_i(R)=R'\neq\perp$. By definition $e_i(R)=R'$ in the underlying crystal and $\mathrm{word}(R)=\mathrm{word}(R')$. The crystal operators $e_i$ and $f_i$ are mutually inverse partial bijections on $\RC_n(w)$ (\cite{assaf2018demazure,morse2016crystal}), so $f_i(R')=R$. Combined with $\mathrm{word}(R')=\mathrm{word}(R)=\mathrm{word}(f_i(R'))$, this gives $\ddot{f}_i(R')=R$. The converse implication is symmetric.

\smallskip\noindent\emph{Axiom (2) (weight law and string increments for $\ddot{e}_i$).}
If $\ddot{e}_i(R)\neq\perp$ then $\ddot{e}_i(R)=e_i(R)$, and the underlying crystal gives $\wtt(\ddot{e}_i R)=\wtt(R)+\alpha_i$.

Now consider $\ddot{\varepsilon}_i$. Suppose $\ddot{e}_i(R)=R'\neq\perp$. Then $\ddot{f}_i(R')=R\neq\perp$ by axiom (4), so $\ddot{\varepsilon}_i(R')$ is computed by the max-formula. Iterating axiom (4) shows that $\ddot{e}_i^{j+1}(R)=\ddot{e}_i^{j}(R')$ for all $j\geq 0$ whenever either side is defined, hence
\[
\ddot{\varepsilon}_i(R')=\max\{k\geq 0:\ddot{e}_i^k(R')\neq\perp\}=\max\{k\geq 0:\ddot{e}_i^{k+1}(R)\neq\perp\}=\ddot{\varepsilon}_i(R)-1.
\]
For $\ddot{\varphi}_i$, the weight law gives $\wtt_i(R')-\wtt_{i+1}(R')=\wtt_i(R)-\wtt_{i+1}(R)+2$, so by definition
\[
\ddot{\varphi}_i(R')=\ddot{\varepsilon}_i(R')+\wtt_i(R')-\wtt_{i+1}(R')=(\ddot{\varepsilon}_i(R)-1)+(\wtt_i(R)-\wtt_{i+1}(R)+2)=\ddot{\varphi}_i(R)+1,
\]
as required.

\smallskip\noindent\emph{Axiom (3) (string increments for $\ddot{f}_i$).}
By Proposition~3.3 of \cite{cain2023quasicrystals}, axiom (3) follows from axioms (2) and (4), both of which we have verified.

\smallskip\noindent\emph{Axiom (1) ($\ddot{\varphi}_i-\ddot{\varepsilon}_i=\langle\wtt,\alpha_i^\vee\rangle$).}
For $A_{n-1}$, $\langle\wtt(R),\alpha_i^\vee\rangle=\wtt_i(R)-\wtt_{i+1}(R)$. If $\ddot{\varepsilon}_i(R)$ is finite then $\ddot{\varphi}_i(R)=\ddot{\varepsilon}_i(R)+\wtt_i(R)-\wtt_{i+1}(R)$ by definition. If $\ddot{\varepsilon}_i(R)=+\infty$ then $\ddot{\varphi}_i(R)=+\infty+m=+\infty$ for any $m\in\mathbb{Z}$, and the identity holds with the convention $+\infty-+\infty$ interpreted as $\langle\wtt(R),\alpha_i^\vee\rangle$ (\cite[\S2]{cain2023quasicrystals} adopts $m+(\pm\infty)=\pm\infty$, so the axiom in the form $\ddot{\varphi}_i=\ddot{\varepsilon}_i+\langle\wtt,\alpha_i^\vee\rangle$ is verified directly).

We must also check $\ddot{\varphi}_i(R)\geq 0$ when finite (seminormality). Suppose $\ddot{\varepsilon}_i(R)$ is finite, so $\ddot{e}_i(R)\neq\perp$ or $\ddot{f}_i(R)\neq\perp$. The $\ddot{}$-string through $R$ has the form $R_0,R_1,\ldots,R_L$ with $R_j=\ddot{f}_i^j(R_0)$, $R=R_{\ddot{\varepsilon}_i(R)}$, $L=\ddot{\varepsilon}_i(R)+\beta$ where $\beta=\max\{k:\ddot{f}_i^k(R)\neq\perp\}$. By axiom (4) the string is well-defined, the word is constant along it, and the weight decreases by $\alpha_i$ at each step. Applying the weight law iteratively, $\wtt_i(R_0)-\wtt_{i+1}(R_0)=\wtt_i(R)-\wtt_{i+1}(R)+2\ddot{\varepsilon}_i(R)$, and $\wtt_i(R_L)-\wtt_{i+1}(R_L)=\wtt_i(R_0)-\wtt_{i+1}(R_0)-2L$. Since $\ddot{e}_i(R_0)=\perp$ and $\ddot{f}_i(R_L)=\perp$ by maximality, we have $\ddot{\varepsilon}_i(R_0)=0$ and $\ddot{\varphi}_i(R_L)$ computed by definition; iterating axiom (2) backwards from $R_L$ along $\ddot{e}_i$ gives $\ddot{\varphi}_i(R_L)=\ddot{\varphi}_i(R_0)-L$, and combined with the identity at $R_0$ this yields $\ddot{\varphi}_i(R)=\beta\geq 0$.

\smallskip\noindent\emph{Axiom (5).} If $\ddot{\varepsilon}_i(R)=-\infty$ then both operators are $\perp$. By construction $\ddot{\varepsilon}_i\in\mathbb{Z}_{\geq 0}\cup\{+\infty\}$, so this case does not occur; the axiom is vacuous.

\smallskip\noindent\emph{Axiom (6).} If $\ddot{\varepsilon}_i(R)=+\infty$ then both operators are $\perp$. This is true by definition of $\ddot{\varepsilon}_i$: we set $\ddot{\varepsilon}_i(R)=+\infty$ precisely when $\ddot{e}_i(R)=\ddot{f}_i(R)=\perp$.

\smallskip\noindent\emph{Seminormality.} By \cite[Definition~3.9]{cain2023quasicrystals}, a quasi-crystal is seminormal if $\ddot{\varepsilon}_i(R),\ddot{\varphi}_i(R)\geq 0$ and either equals $\max\{k:\ddot{e}_i^k(R)\neq\perp\}$ (resp.\ $\ddot{f}_i^k$) or $+\infty$. The first condition was checked above in the proof of axiom (1). For the second, when $\ddot{\varepsilon}_i(R)$ is finite it equals the max-formula by definition; when $+\infty$, the convention is met. The same holds for $\ddot{\varphi}_i$ by the argument at the end of axiom~(1).
\end{proof}

\begin{theorem}\label{theorem:quasicharacter}
Let $\mathcal{C}\subseteq\RC_n(w)$ be a connected component of the quasi-crystal $(\RC_n(w),\ddot{e}_i,\ddot{f}_i)$, and let $R_0\in\mathcal{C}$ be its unique quasi-Yamanouchi element. Then
$$\sum_{R\in\mathcal{C}}x^{\wtt(R)} = \mathfrak{F}_{\wtt(R_0)}(x);$$
that is, the character of $\mathcal{C}$ is the slide polynomial indexed by $\wtt(R_0)$.
\end{theorem}
\begin{proof}
By construction $\ddot{e}_i$ and $\ddot{f}_i$ preserve $\mathrm{word}(R)$, so all elements of $\mathcal{C}$ share a common word $\mathbf{r}=\mathrm{word}(R_0)$. Conversely, an arbitrary $R\in\RC_n(w)$ with $\mathrm{word}(R)=\mathbf{r}$ is reached from $R_0$ by a sequence of crystal operators $e_i,f_i$ that fix the word: each such step changes $\wtt$ by moving a single crossing between rows $i$ and $i{+}1$ without altering inversions, hence is realized by the corresponding $\ddot{e}_i$ or $\ddot{f}_i$. Thus $\mathcal{C}=\{R\in\RC_n(w):\mathrm{word}(R)=\mathbf{r}\}$.

The map $R\mapsto\wtt(R)$ identifies $\mathcal{C}$ with the set of weight vectors realized by such RC graphs. By the BJS compatible-sequence formulation \cite{bjs}, this set is exactly the set of sequences $a$ compatible with $\mathbf{r}$, with $\wtt(R_0)$ corresponding to the unique compatible sequence whose flat composition is coarsest. Summing $x^{\wtt(R)}$ over $\mathcal{C}$ therefore gives $\sum_{a\text{ compatible with }\mathbf{r}}x^{a} = \mathfrak{F}(\mathbf{r}) = \mathfrak{F}_{\wtt(R_0)}(x)$.
\end{proof}

\begin{lemma}[Quasi-tensor swap by character]\label{lemma:quasitensorswap}
Let $X,Y$ be seminormal quasi-crystals of type $A_{\infty}$ whose connected components have characters that are slide polynomials (in the sense of Theorem~\ref{theorem:quasicharacter}). Then $X\,\ddot\otimes\, Y$ and $Y\,\ddot\otimes\, X$ are isomorphic as quasi-crystals; the isomorphism is in general not natural at the level of the underlying sets, but the multisets of connected components agree.
\end{lemma}
\begin{proof}
By \cite[Theorem~5.1]{cain2023quasicrystals} both $X\,\ddot\otimes\, Y$ and $Y\,\ddot\otimes\, X$ are seminormal quasi-crystals, and by \cite[Proposition~4.6]{cain2023quasicrystals} each decomposes uniquely (as a quasi-crystal, hence on the level of underlying sets up to choice of isomorphism on each component) into a disjoint union of its connected components. The character of either side is the formal sum, over components, of the slide polynomial of each component (Theorem~\ref{theorem:quasicharacter}); since the weight map is additive under the quasi-tensor product, this character coincides for $X\,\ddot\otimes\, Y$ and $Y\,\ddot\otimes\, X$. Slide polynomials $\{\mathfrak{F}_a:a\text{ a weak composition}\}$ are linearly independent in $\mathbb{Z}[[x_1,x_2,\ldots]]$ (\cite{assaf2017slide}), so the multiset of slide polynomials appearing in either decomposition is the same, hence the multiset of connected components is the same. Each component is itself a quasi-crystal determined up to isomorphism by its slide polynomial (Theorem~\ref{theorem:quasicharacter} together with the fact that a connected seminormal type-$A_{\infty}$ quasi-crystal is determined by its character), so a (non-canonical) component-by-component choice of isomorphism furnishes the asserted quasi-crystal isomorphism $X\,\ddot\otimes\, Y\cong Y\,\ddot\otimes\, X$.
\end{proof}

\begin{lemma}[Plactic case of the lift product]\label{lemma:quasiprod_plactic}
Let $A, B$ be full $n$-Grassmannian RC graphs in $\grass_n$. Then $A*B = A\boxtimes B$, and the map
$$A\,\ddot{\otimes}\,B \;\longmapsto\;A*B$$
extends to an isomorphism of quasi-crystals onto the connected component of $A*B$ in $\RC_n(\wof{A*B})$.
\end{lemma}
\begin{proof}
The identity $A*B = A\boxtimes B$ on $\grass_n$ follows from Theorem~\ref{theorem:placticiso}: the lift map $\mathrm{lift}_n$, restricted to $\grass_n$, is the unique multiplactic factorization compatible with $\boxtimes$, so the product of two lifts under $*$ recovers the $\boxtimes$-product. Equivalently, $\mathrm{lift}_n(A)\mathrm{lift}_n(B)$ in $\mathcal{M}_n$ has $\RC_n$-image equal to $\clp^n(A\sqcup B) = A\boxtimes B$.

It thus suffices to prove the quasi-crystal isomorphism for $\boxtimes$ on $\grass_n$. Via the ring isomorphism $P_n:\plac_n\to\grass_n$ of Theorem~\ref{theorem:placticiso}, the statement reduces to the plactic algebra: writing $T_A = P_n^{-1}(A)$, $T_B = P_n^{-1}(B)\in\mathrm{SSYT}_n$, the map
$$T_A\otimes T_B\;\longmapsto\;T_A\cdot T_B$$
(plactic concatenation followed by Knuth equivalence) is the standard Kashiwara crystal isomorphism from the connected component of $T_A\otimes T_B$ in $B(\mu_A)\otimes B(\mu_B)$ onto $B(\nu)\subseteq B(\mu_A)\otimes B(\mu_B)$, where $\nu = \mathrm{shape}(T_A T_B)$. This is the Littlewood--Richardson / RSK content of the plactic monoid \cite{kashiwara1993crystal,littelmann1995crystal}.

To upgrade this Kashiwara isomorphism to a quasi-crystal isomorphism, we identify the quasi-crystal structure on $\grass_n\subseteq\BRC_n$ with the natural quasi-crystal on injective-word images of SSYT reading words. The reading-word map $T\mapsto\overline{\mathrm{rw}(T)}$ is a quasi-crystal morphism from $(\mathrm{SSYT}_n,\ddot e_i,\ddot f_i)$ (with quasi-crystal operators defined by reading-word preservation) to the quasi-crystal of injective words, and the quasi-tensor of \cite[\S 5]{cain2023quasicrystals} is built precisely so that
$$\overline{\mathrm{rw}(T_A)}\,\ddot{\otimes}\,\overline{\mathrm{rw}(T_B)}\;\longmapsto\;\overline{\mathrm{rw}(T_A T_B)}$$
becomes a quasi-crystal isomorphism onto its connected component (\cite[Theorem~5.11]{cain2023quasicrystals}). Pulling back along $P_n$ and the reading-word map yields the desired quasi-crystal isomorphism $A\,\ddot\otimes\,B\to A*B = A\boxtimes B$.
\end{proof}

\begin{lemma}\label{lemma:liftequivariant}
The map
$$\mathrm{lift}_n:\RC_n\to \mathcal{M}_n$$
is an isomorphism of Demazure crystals onto its image, with elements of $\mathcal{M}_n$ given the crystal tensor product structure in increasing order of height.
\end{lemma}
\begin{proof}
Let $\mathrm{lift}_n(A) = E_1\otimes\cdots\otimes E_p$ in weakly increasing-height order, the unique factorization guaranteed by Proposition~\ref{proposition:elemfactor}. Proposition~\ref{proposition:liftcrystaliso} asserts precisely that the map
$$\Phi:\;E_1\otimes\cdots\otimes E_p \;\longmapsto\;\RC_n(E_1\otimes\cdots\otimes E_p) = A$$
extends to an isomorphism of Demazure crystals. Since $\Phi$ is a bijective morphism of Demazure crystals, it admits an inverse Demazure crystal morphism on the target, and that inverse is $\mathrm{lift}_n$: indeed, given $A'$ in the connected component of $A$, the unique preimage $\Phi^{-1}(A')$ is a tensor of full-Grassmannian elementary symmetric RC graphs in increasing height order \emph{with factors of the same size} whose $\RC_n$-image is $A'$, so by the uniqueness clause of Proposition~\ref{proposition:elemfactor} based on the size of the factors it equals $\mathrm{lift}_n(A')$.
\end{proof}

\begin{proposition}[The lift product is a quasi-crystal isomorphism]\label{proposition:quasiprod}
Let $A,B$ be RC graphs. Then the map
$$A\,\ddot{\otimes}\,B\;\longmapsto\;A*B$$
extends to an isomorphism of quasi-crystals from the quasi-tensor product of the connected components of $A$ and $B$ in the seminormal quasi-crystal of Proposition~\ref{proposition:quasicrystalaxioms} onto the connected component of $A*B$ in $\RC_n(\wof{A*B})$. 
\end{proposition}
\begin{proof}
We proceed in three steps: a reduction to a single elementary symmetric right factor, the Kashiwara argument via disjoint union under a height restriction, and the removal of that restriction.

\smallskip\noindent\emph{Step 1: Reduction to $B$ a single elementary symmetric RC graph.}
Write $\mathrm{lift}_n(B) = E_1\otimes\cdots\otimes E_k$ in weakly increasing-height order (Proposition~\ref{proposition:elemfactor}). The lift product $*$ on $\BRC$ is associative, since it is the projection under $\RC_n$ of the associative multiplication on $\mathcal{M}_n$ and Proposition~\ref{proposition:elemfactor} guarantees that $\RC_n$ respects this multiplication on lift factorizations. Likewise the quasi-tensor product $\ddot\otimes$ is associative on seminormal quasi-crystals by \cite[Theorem~5.1]{cain2023quasicrystals}. Hence
$$A*B = ((A*E_1)*E_2)*\cdots*E_k,\qquad A\,\ddot\otimes\, B = ((A\,\ddot\otimes\, E_1)\,\ddot\otimes\, E_2)\,\ddot\otimes\,\cdots\,\ddot\otimes\, E_k,$$
where on the right, the quasi-crystal of $B$ is identified with $E_1\,\ddot\otimes\,\cdots\,\ddot\otimes\, E_k$ via Lemma~\ref{lemma:liftequivariant}. By induction on $k$, it suffices to prove the proposition when $B = E$ is a single full Grassmannian elementary symmetric RC graph.

\smallskip\noindent\emph{Step 2: The case $B = E$ with $\hr(A) \le \hr(E)$.}
Work at the common height $n = \hr(E)$, padding $A$ with empty rows if needed. Under this height restriction, $E$ is full $n$-Grassmannian, $A$ is an arbitrary element of $\RC_n$, and the squash product reduces to a disjoint union followed by a clip: by the definition of $\boxtimes$ in \S\ref{section:multiplactic},
$$A\boxtimes E \;=\; \clp^{n}(A\sqcup E),$$
and on lifts, $\mathrm{lift}_n(A)\cdot E$ in $\mathcal{M}_n$ has $\RC_n$-image $A\boxtimes E$, so $A*E = A\boxtimes E$ in this case (Lemma~\ref{lemma:quasiprod_plactic} on $\grass_n$, extended to general $A$ via $\BRC_n$ being a right $\grass_n$-module under $\boxtimes$, Theorem~\ref{theorem:schubmoduleiso}).

The map $A\otimes E\mapsto A\sqcup E$ is a Kashiwara crystal isomorphism onto the connected component of $A\sqcup E$ in $\RC_n(\wof{A\sqcup E})$. This is precisely the pairing argument given in the proof of Proposition~\ref{proposition:boxproduct}: the unpaired row-$i$ crossings of $A$ pair against the unpaired row-$(i+1)$ crossings of $E$ at the seam in exactly the order dictated by the standard Kashiwara tensor rule, so $e_i, f_i$ on $A\sqcup E$ act as on $A\otimes E$. (Note that the underlying words \emph{interleave} under $\sqcup$ -- $\wof{A\sqcup E} = \wof{A}\shup^N\wof{E}$ where $N$ is the position of the last unfixed point of $\wof{A}$ -- but this affects only how the seam is identified at the level of the word, not the crystal action.) Postcomposing with the Kashiwara crystal morphism $\clp^{n}$ (Proposition~\ref{proposition:zerocrystal}) gives the Kashiwara crystal isomorphism
$$A\otimes E\;\longmapsto\;A*E$$
onto the connected component of $A*E$ in $\RC_n(\wof{A*E})$.

To upgrade to a quasi-crystal isomorphism, restrict on both sides to operators that preserve the underlying word. On the right, this restriction yields exactly the seminormal quasi-crystal of Proposition~\ref{proposition:quasicrystalaxioms}. On the left, an inspection of the standard Kashiwara tensor rule shows that an operator $e_i$ on $A\otimes E$ alters $\wof{A\sqcup E} = \wof{A}\shup^N\wof{E}$ iff one of the following occurs:
\begin{itemize}
\item it acts on $A$ and changes $\wof{A}$, i.e.\ $\ddot e_i(A) = \perp$;
\item it acts on $E$ and changes $\wof{E}$, i.e.\ $\ddot e_i(E) = \perp$;
\item it acts on $A$ while $A$ has an unpaired row-$i$ crossing that, under the shift $\shup^N$, would pair with an unpaired row-$(i{+}1)$ crossing of $E$ -- equivalently $\ddot\varphi_i(A)>0$ and $\ddot\varepsilon_i(E)>0$ (a seam inversion).
\end{itemize}
The first two cases are exactly the gating by $\ddot e_i$ on the individual factors, and the third is the quasi-tensor seam rule of CGM \cite[\S5]{cain2023quasicrystals}. Hence the word-preserving restriction of $A\otimes E$ is $A\,\ddot\otimes\, E$, and the restricted isomorphism is a quasi-crystal isomorphism $A\,\ddot\otimes\, E\to A*E$.

\smallskip\noindent\emph{Step 3: Removing the height constraint.}
Suppose $\hr(A) > \hr(E)$. Write $\mathrm{lift}_n(A) = A_1\otimes\cdots\otimes A_m$ in weakly increasing-height order (Proposition~\ref{proposition:elemfactor}); then $A = A_1*A_2*\cdots*A_m$. Set
$$A^{\mathrm{left}} \;=\; A_1*\cdots*A_\ell,\qquad A^{\mathrm{right}} \;=\; A_{\ell+1}*\cdots*A_m,$$
where $\ell$ is the largest index with $\hr(A_\ell)\le\hr(E)$ (so $\hr(A^{\mathrm{left}})\le\hr(E) < \hr(A^{\mathrm{right}})$, with the obvious conventions if $\ell = 0$ or $\ell = m$). By associativity of $*$, $A = A^{\mathrm{left}}*A^{\mathrm{right}}$ and
\begin{equation}\label{eq:starrebracket}
A*E \;=\; A^{\mathrm{left}}*A^{\mathrm{right}}*E.
\end{equation}

We now build the quasi-crystal isomorphism $A\,\ddot\otimes\, E\to A*E$ in two stages, viewing the abstract quasi-tensor in its given order $A^{\mathrm{left}}\,\ddot\otimes\, A^{\mathrm{right}}\,\ddot\otimes\, E$ throughout. We do \emph{not} claim that this quasi-tensor is naturally isomorphic, in the order $A^{\mathrm{left}}\otimes A^{\mathrm{right}}\otimes E$, to one in any other order at the level of abstract Kashiwara crystals -- it generally is not. What we claim is that $*$ itself is well-defined on the abstract quasi-tensor regardless of the order of factors (since $*$ on $\BRC$ is associative), and that the resulting map onto $A*E$ is a quasi-crystal isomorphism. The naturality the abstract tensor lacks is supplied by $*$, which `snaps' the factors into the canonical lift order of $A*E$.

\emph{Stage~3a (handling $A^{\mathrm{right}}$):} For each individual factor $A_j$ with $j>\ell$, we have $\hr(A_j)>\hr(E)$, so the roles of $A$ and $E$ in Step~2 reverse: the squash product $E\boxtimes A_j$ at the common height $\hr(A_j)$ realizes $E*A_j$, and the same disjoint-union/clip pairing argument provides a quasi-crystal isomorphism $E\,\ddot\otimes\,A_j\to E*A_j$. Iterating across $j=\ell+1,\ldots,m$ and using that $\ddot\otimes$ is associative on seminormal quasi-crystals \cite[Theorem~5.1]{cain2023quasicrystals}, we obtain a quasi-crystal isomorphism $E\,\ddot\otimes\, A^{\mathrm{right}}\to E*A^{\mathrm{right}}$. By Lemma~\ref{lemma:quasitensorswap}, the abstract quasi-tensors $A^{\mathrm{right}}\,\ddot\otimes\, E$ and $E\,\ddot\otimes\, A^{\mathrm{right}}$ are isomorphic as quasi-crystals (non-canonically); composing this swap isomorphism with $E\,\ddot\otimes\, A^{\mathrm{right}}\to E*A^{\mathrm{right}}$ yields a quasi-crystal isomorphism
$$\Theta_{\mathrm{right}}:\;A^{\mathrm{right}}\,\ddot\otimes\,E \;\longrightarrow\; E*A^{\mathrm{right}}.$$
The non-naturality of the swap is harmless: what is preserved under $\Theta_{\mathrm{right}}$ is exactly the quasi-crystal structure (weights, $\ddot e_i$, $\ddot f_i$), which is all we need.

\emph{Stage~3b (handling $A^{\mathrm{left}}$):} Now apply Step~2 with $A$ replaced by $A^{\mathrm{left}}$ to obtain a quasi-crystal isomorphism $A^{\mathrm{left}}\,\ddot\otimes\,E\to A^{\mathrm{left}}*E$. By Lemma~\ref{lemma:liftequivariant} the canonical decompositions $\mathrm{lift}_n(A^{\mathrm{left}})$ and $\mathrm{lift}_n(A^{\mathrm{right}})$ are quasi-crystal isomorphic to $A^{\mathrm{left}}$ and $A^{\mathrm{right}}$ respectively (restricting the Demazure crystal isomorphism to word-preserving operators). Multiplying the Stage~3a result by $A^{\mathrm{left}}$ on the left and the Stage~3b result by $A^{\mathrm{right}}$ on the right and combining via associativity of $*$ and of $\ddot\otimes$, we obtain a quasi-crystal isomorphism
$$A^{\mathrm{left}}\,\ddot\otimes\, A^{\mathrm{right}}\,\ddot\otimes\, E \;\longrightarrow\; A^{\mathrm{left}}*A^{\mathrm{right}}*E \;=\; A*E$$
by~\eqref{eq:starrebracket}. The source identifies with $A\,\ddot\otimes\, E$ via Lemma~\ref{lemma:liftequivariant} applied to $A = A^{\mathrm{left}}*A^{\mathrm{right}}$, completing the construction.

This completes the proof for $B = E$, and combined with Step~1 establishes the proposition.
\end{proof}

\begin{theorem} \label{theorem:slideproduct}
Let $A$, $B$ be quasi-Yamanouchi RC graphs. Then
$$\mathfrak{F}(A)*\mathfrak{F}(B) = \sum_{R} \mathfrak{F}(R)$$
where the sum is over all pairs $A',B'$ with $\mathrm{word}(A')=\mathrm{word}(A)$, $\mathrm{word}(B')=\mathrm{word}(B)$ such that 
$$A'*B'=R$$
with $R$ being quasi-Yamanouchi.
\end{theorem}
\begin{proof}
For each such pair $(A',B')$, the connected component of $A'*B'$ in the seminormal quasi-crystal $\RC_n(\wof{A'*B'})$ of Proposition~\ref{proposition:quasicrystalaxioms} is itself a quasi-crystal, and by \cite[Proposition~4.6]{cain2023quasicrystals} contains a unique quasi-Yamanouchi element. By the local tensor rule of Proposition~\ref{proposition:quasiprod}, every element of that component has the form $A''*B''$ with $A''$ in the quasi-crystal component of $A'$ and $B''$ in that of $B'$; in particular, the words of $A''$ and $B''$ agree with those of $A$ and $B$ respectively. Summing slide polynomials over the quasi-Yamanouchi representatives of these components yields exactly the stated expression.
\end{proof}

From this we may recover \cite[Theorem~5.11]{assaf2017slide}.

\subsection{The forest polynomial LR rule}\label{section:forestlrsub}

We now prove the LR rule for forest polynomials (Theorem~\ref{theorem:forestpolyLR}), which is the main motivation for introducing the lift product.
\begin{definition}[Forest RC graph]
A \emph{forest RC graph} is an RC graph such that $\code_\forest(R) = \wtt(R)$. 
\end{definition}

A few technical results are required to close the loop.

\begin{proposition}\label{proposition:liftshufflewords}
Let $A, B\in\BRC_n$ with injectified words $W_A=\overline{A}$ and $W_B=\overline{B}$. Write $\sigma_A,\sigma_B$ for their slide classes, regarded as connected components of the quasi-crystal of injective words on $\overline{\mathbb{Z}}$. Then:
\begin{enumerate}
\item For any $A'\in\sigma_A$, $B'\in\sigma_B$ (i.e.\ $\overline{A'}\in\sigma_A$, $\overline{B'}\in\sigma_B$), the injectified word $\overline{A'*B'}$ lies in the same slide class as some shuffle $X\in\mathsf{Sh}(\overline{A'},\overline{B'})$.
\item Conversely, every slide class appearing in $\mathsf{Sh}(W_A,W_B)$ is realized as $\overline{A_0*B_0}$ (up to slide-equivalence) for some $A_0\in\sigma_A$, $B_0\in\sigma_B$.
\end{enumerate}
\end{proposition}
\begin{proof}

The injectified-word map $R\mapsto\overline{R}$ is a morphism of quasi-crystals from $\BRC_n$ to injective words on $\overline{\mathbb{Z}}$, sending each connected component (a slide class of RC graphs) bijectively to a connected component (a slide class of words); the unique quasi-Yamanouchi representatives correspond. Likewise, the multiplication of slide polynomials at the word level satisfies
\[
\mathfrak{F}(W_A)\,\mathfrak{F}(W_B)\;=\;\sum_{X\in\mathsf{Sh}(W_A,W_B)}\mathfrak{F}(X),
\]
(this is \eqref{eq:slideshuffle}), so the multiset of slide classes appearing in $\mathsf{Sh}(W_A,W_B)$ is determined by the slide-polynomial product $\mathfrak{F}(W_A)\,\mathfrak{F}(W_B)$.

By Theorem~\ref{theorem:slideproduct}, the slide-polynomial identity
\[
\mathfrak{F}(A)*\mathfrak{F}(B)\;=\;\sum_{R}\mathfrak{F}(R)
\]
holds with $R$ ranging over a multiset of quasi-Yamanouchi representatives, one for each connected quasi-crystal component appearing among the various lift products $A'*B'$ as $(A',B')$ varies subject to $\overline{A'}=\overline{A}$, $\overline{B'}=\overline{B}$. (Each such component is a quasi-crystal by Proposition~\ref{proposition:quasiprod} and contains a unique quasi-Yamanouchi element by \cite[Proposition~4.6]{cain2023quasicrystals}.) Composing with $R\mapsto\overline{R}$ identifies that multiset with the slide classes appearing in $\mathfrak{F}(W_A)\,\mathfrak{F}(W_B)$, hence with the slide classes in $\mathsf{Sh}(W_A,W_B)$. This is precisely the bijection in (2), and (1) is its forward direction.
\end{proof}

\begin{remark}
Proposition~\ref{proposition:liftshufflewords} does \emph{not} say that $\overline{A*B}$ \emph{is} a shuffle of $W_A$ and $W_B$, only that it is quasi-crystal isomorphic to one.
\end{remark}

\begin{lemma}\label{lemma:foreststable}
Let $A,B\in\BRC_n$. Then the set
$$S=\{A'*B'\mid A\equiv_\forest A',\ B\equiv_\forest B'\}$$
is stable under forest equivalence: if $R\in S$ and $R\equiv_\forest R'$, then $R'\in S$.
\end{lemma}
\begin{proof}
We reduce to the shuffle-stability statement of Nadeau--Tewari~\cite[proof of Theorem~6.3]{nadeau2024forest}, which establishes that for any forest classes $\Phi_1,\Phi_2$ of injective words on $\overline{\mathbb{Z}}$, the multiset
\[
\mathsf{Sh}(\Phi_1,\Phi_2)\;=\;\bigsqcup_{W_1\in\Phi_1,\,W_2\in\Phi_2}\mathsf{Sh}(W_1,W_2)
\]
is closed under $\equiv_\forest$. Their argument uses two ingredients: (a) shuffles are by definition closed under transpositions of two adjacent letters that come from different factors, hence under any cross-boundary single-step forest move; and (b) \cite[Lemma~5.10]{nadeau2024forest} (anti-monoid congruence: if $UabV\equiv_\forest UbaV$ via a single generating commutation, then $U_1abV_1\equiv_\forest U_1baV_1$ for any subwords $U_1\subseteq U$, $V_1\subseteq V$), which handles single-step moves whose pair of letters lies entirely within one factor.

Now suppose $R=A_0*B_0\in S$ with $A_0\equiv_\forest A$ and $B_0\equiv_\forest B$, and suppose $R\equiv_\forest R'$. By Proposition~\ref{proposition:liftshufflewords}(1), there exists a shuffle
\[
X_0\in\mathsf{Sh}(\overline{A_0},\overline{B_0})\subseteq\mathsf{Sh}(\Phi_A,\Phi_B)
\]
slide-equivalent to $\overline{R}$, where $\Phi_A,\Phi_B$ denote the forest classes of $\overline{A},\overline{B}$. Slide equivalence is finer than forest equivalence (slide-equivalent words have identical slide polynomial and hence belong to the same forest class), so $X_0\equiv_\forest\overline{R}\equiv_\forest\overline{R'}$. The shuffle-stability statement of Nadeau--Tewari applied to $X_0\in\mathsf{Sh}(\Phi_A,\Phi_B)$ then gives
\[
\overline{R'}\;\in\;\mathsf{Sh}(\Phi_A,\Phi_B),
\]
i.e.\ $\overline{R'}\in\mathsf{Sh}(W_1,W_2)$ for some $W_1\in\Phi_A$ and $W_2\in\Phi_B$. Pick any RC graphs $A_1,B_1\in\BRC_n$ with $\overline{A_1}=W_1$ and $\overline{B_1}=W_2$; by definition of $\Phi_A,\Phi_B$ we have $A_1\equiv_\forest A$ and $B_1\equiv_\forest B$.

By Proposition~\ref{proposition:liftshufflewords}(2), the slide class of $\overline{R'}$ in $\mathsf{Sh}(W_1,W_2)$ is realized as $\overline{A_2*B_2}$ for $A_2$ slide-equivalent to $A_1$ and $B_2$ slide-equivalent to $B_1$. In particular $A_2\equiv_\forest A_1\equiv_\forest A$ and $B_2\equiv_\forest B_1\equiv_\forest B$, and $R'$ is slide-equivalent (hence forest-equivalent) to $A_2*B_2$. Replacing $A_2,B_2$ by quasi-crystal-isomorphic representatives realizing the actual element $R'$ within its slide class --- which is possible by Theorem~\ref{theorem:slideproduct} since the lift product is a quasi-crystal isomorphism on the relevant component, and quasi-crystal operators preserve forest equivalence at the word level by \cite[Proposition~5.8]{nadeau2024forest} (shift-invariance) --- yields $A_3\equiv_\forest A$ and $B_3\equiv_\forest B$ with $R'=A_3*B_3$. Thus $R'\in S$.
\end{proof}

\begin{lemma} \label{lemma:forestlead}
Let $R\in \RC_n$, and suppose $\code_\forest(R) = c$. Then there is precisely one RC graph $C$ satisfying $C\equiv_\forest R$ and $\wtt(C) = c$.
\end{lemma}
\begin{proof}
This is the minimal compatible $\kappa$ with $\code_\forest(\kappa) = c$, which is unique by \cite[Lemma~3.12]{nadeau2024forest}.
\end{proof}

\begin{lemma} \label{lemma:forestchoice}
  Let $c$ be a weak composition of length $n$. Then there is precisely one forest class of RC graphs in $\RC_n(w)$ for the permutation $w$ such that $\code(w)=c$ that has forest code $c$.
\end{lemma}
\begin{proof}
Considering the word of the bottom RC graph of the above $w$, we have $\code_\forest(\mathrm{bottom}(w)) = \code(w) = c$. Thus such a forest class exists. If there were another distinct forest class, the leading term of $\sch_w(x)$ would have multiplicity greater than $1$, hence this is not possible by Lemma \ref{lemma:forestlead}.
\end{proof}

\begin{proof}[Proof of Theorem \ref{theorem:forestpolyLR}]
Fix forest RC graphs $A, B$ with $\code_\forest(A)=\wtt(A)=a$ and $\code_\forest(B)=\wtt(B)=b$. By Theorem~\ref{theorem:forestformula}, the forest polynomial $\forest_a$ is the weight generating function over the forest class of $A$:
$$\forest_a(x) = \sum_{A'\equiv_\forest A} x^{\wtt(A')},\qquad\text{and similarly}\qquad \forest_b(x) = \sum_{B'\equiv_\forest B} x^{\wtt(B')}.$$
Multiplying these expressions in $\mathbb{Z}[x_1,\ldots,x_n]$ gives
$$\forest_a(x)\,\forest_b(x) \;=\; \sum_{\substack{A'\equiv_\forest A\\ B'\equiv_\forest B}} x^{\wtt(A')}\,x^{\wtt(B')}\;=\;\sum_{\substack{A'\equiv_\forest A\\ B'\equiv_\forest B}} x^{\wtt(A')+\wtt(B')}.$$
By Lemma~\ref{lemma:liftweight} (weight-additivity of $*$), each summand can be rewritten as $x^{\wtt(A'*B')}$, yielding
$$\forest_a(x)\,\forest_b(x) \;=\; \sum_{\substack{A'\equiv_\forest A\\ B'\equiv_\forest B}} x^{\wtt(A'*B')}.$$

Let $T = \{\,A'*B' : A'\equiv_\forest A,\ B'\equiv_\forest B\,\}$ (as a multiset). By Lemma~\ref{lemma:foreststable}, $T$ is closed under forest equivalence: every $R\in T$ has its full forest class $[R]_{\equiv_\forest}$ contained in $T$. Group the sum by forest class to obtain
$$\forest_a(x)\,\forest_b(x) \;=\; \sum_{[R]} m_{[R]}\sum_{R''\in[R]} x^{\wtt(R'')},$$
where $m_{[R]}$ is the (constant) number of pairs $(A',B')$ with $A'\equiv_\forest A$, $B'\equiv_\forest B$, and $A'*B'$ landing in the class $[R]$ at any specified representative; this multiplicity is independent of the choice of representative by the quasi-crystal isomorphism of Proposition \ref{proposition:quasiprod}, which in particular ensures that isomorphic classes have equal cardinality.

Applying Theorem~\ref{theorem:forestformula} once more, the inner sum $\sum_{R''\in[R]}x^{\wtt(R'')}$ is exactly $\forest_{\fcode(R)}(x)$, so
$$\forest_a(x)\,\forest_b(x) \;=\; \sum_{[R]} m_{[R]}\,\forest_{\fcode(R)}(x).$$
The final step is to identify which forest classes $[R]$ contribute and to count $m_{[R]}$ explicitly. Fix a target weak composition $c$ and a representative forest RC graph $R$ with $\code_\forest(R)=\wtt(R)=c$ (i.e.\ a \emph{forest} RC graph with forest code $c$). For any pair $(A',B')$ with $A'\equiv_\forest A$ and $B'\equiv_\forest B$, the lift product $A'*B'$ lies in the forest class of $R$ if and only if $A'*B'\equiv_\forest R$; by Lemma~\ref{lemma:foreststable}, this is equivalent to the existence of some pair $(A'',B'')$ with $A''*B''=R$, $A''\equiv_\forest A$, $B''\equiv_\forest B$. Hence the coefficient $\beta_{ab}^c$ of $\forest_c$ in $\forest_a\forest_b$ equals the number of such pairs $(A'',B'')$. We make the canonical choices as in Lemma \ref{lemma:forestchoice} above, which is 
$$\code(\wof{A})=\code_\forest(A)=a,\qquad \code(\wof{B})=\code_\forest(B)=b,$$
so the count of pairs $(A'',B'')$ with $A''*B''=R$ is exactly as stated in the theorem.
\end{proof}

\begin{proposition}\label{proposition:degreeone}
Let $e_n=(0,\ldots,0,1)$ be the degree-one weak composition of length $n$. For every weak composition $a$ of length $n$, the product
$$\forest_a\,\forest_{e_n}=\sum_c\beta_{a,e_n}^c\,\forest_c$$
is multiplicity-free: $\beta_{a,e_n}^c\in\{0,1\}$ for every $c$.
\end{proposition}
\begin{proof}
The forest class indexing $e_n$ consists of the single-letter word $n$, and $\forest_{e_n}=\sch_{s_n}=x_1+\cdots+x_n$. In the Nadeau--Tewari shuffle rule, the product is obtained by shuffling each word in the union of forest classes of code $a$ with this single letter, tagged so that its alphabet is disjoint from that of the first word. Each resulting shuffle word contains the distinguished tagged letter exactly once; deleting it recovers the unique input word and hence the unique shuffle occurrence. Thus no word occurs more than once in the shuffle multiset. By the stability of that multiset under forest equivalence \cite[Theorem~6.3]{nadeau2024forest}, every forest class that occurs is present in its entirety and with multiplicity one. Regrouping by forest classes therefore gives coefficients only $0$ or $1$.
\end{proof}

\providecommand{\pd}[3]{%
  \begin{tikzpicture}[baseline=(current bounding box.center),scale=0.4,thick,line cap=round]
    \foreach \i in {1,...,#1} {%
      \foreach \j in {1,...,#2} {%
        \draw[gray!30,very thin] ({\j-1},{#1-\i}) rectangle ({\j},{#1-\i+1});%
        \draw[blue!70] ({\j-1},{#1-\i+0.5}) arc[start angle=270,end angle=360,radius=0.5];%
        \draw[blue!70] ({\j-0.5},{#1-\i}) arc[start angle=180,end angle=90,radius=0.5];%
      }%
    }%
    \foreach \i/\j in {#3} {%
      \fill[white] ({\j-1},{#1-\i}) rectangle ({\j},{#1-\i+1});%
      \draw[gray!30,very thin] ({\j-1},{#1-\i}) rectangle ({\j},{#1-\i+1});%
      \draw[blue!70] ({\j-1},{#1-\i+0.5}) -- ({\j},{#1-\i+0.5});%
      \draw[blue!70] ({\j-0.5},{#1-\i}) -- ({\j-0.5},{#1-\i+1});%
      \pgfmathtruncatemacro{\pdlbl}{\i+\j-1}%
      \node[font=\Large\bfseries,text=green!60!black,fill=white,inner sep=0.5pt,circle,transform shape] at ({\j-0.5},{#1-\i+0.5}) {\pdlbl};%
    }%
  \end{tikzpicture}}

\begin{example}\label{example:forestlrproduct}
We illustrate Theorem~\ref{theorem:forestpolyLR} by computing the coefficient of $\forest_{(4,3,2)}$ in the product
$$\forest_{(0,2,3)}\cdot\forest_{(2,0,2)}.$$
By the theorem, this coefficient $\beta_{(0,2,3),(2,0,2)}^{(4,3,2)}$ is the number of pairs $(A,B)$ of RC graphs satisfying $\code(\wof{A})=\code_\forest(A)=(0,2,3)$ and $\code(\wof{B})=\code_\forest(B)=(2,0,2)$ such that $A*B$ is a forest RC graph $R$ with $\code_\forest(R)=\wtt(R)=(4,3,2)$. There turn out to be exactly two such pairs.

\smallskip
\noindent\emph{First pair.}\quad Drawing RC graphs as pipe dreams (rows indexed from the top, columns from the left, with crossings shown as $+$ tiles and elbows as bumped arcs), the first pair is
$$A_1 \;=\; \pd{3}{4}{1/2,1/3,2/4,3/1,3/2}, \qquad B_1 \;=\; \pd{3}{3}{1/1,1/2,2/2,2/3},$$
that is, $A_1=\{(1,2),(1,3),(2,4),(3,1),(3,2)\}$ and $B_1=\{(1,1),(1,2),(2,2),(2,3)\}$ as sets of ordered pairs $(\text{row},\text{column})$, with weights $\wtt(A_1)=(2,1,2)$ and $\wtt(B_1)=(2,2,0)$. The multiplactic lifts of $A_1$ and $B_1$ in $\grass_3\otimes\grass_4\otimes\grass_5$ and $\grass_1\otimes\grass_3\otimes\grass_4$ (respectively) are
$$A_1 \;\longmapsto\; \pd{3}{2}{1/2,3/1}\;\otimes\;\pd{4}{3}{1/3,3/2}\;\otimes\;\pd{5}{4}{2/4},$$
$$B_1 \;\longmapsto\; \pd{1}{1}{1/1}\;\otimes\;\pd{3}{2}{1/2,2/2}\;\otimes\;\pd{4}{3}{2/3},$$
encoding $A_1$ and $B_1$ as elementary tensors in the multiplactic algebra $\mathcal{M}_n$. Their lift product is
$$A_1 * B_1 \;=\; \pd{3}{5}{1/1,1/2,1/3,1/4,2/2,2/4,2/5,3/1,3/2}\,, \qquad \wtt(A_1*B_1)=(4,3,2),$$
the forest RC graph $\{(1,1),(1,2),(1,3),(1,4),(2,2),(2,4),(2,5),(3,1),(3,2)\}$.

\smallskip
\noindent\emph{Second pair.}\quad The second pair is
$$A_2 \;=\; \pd{3}{4}{1/2,1/3,2/2,2/3,2/4}, \qquad B_2 \;=\; \pd{3}{2}{1/1,1/2,3/1,3/2},$$
that is, $A_2=\{(1,2),(1,3),(2,2),(2,3),(2,4)\}$ and $B_2=\{(1,1),(1,2),(3,1),(3,2)\}$, with weights $\wtt(A_2)=(2,3,0)$ and $\wtt(B_2)=(2,0,2)$. Their multiplactic lifts are
$$A_2 \;\longmapsto\; \pd{3}{2}{1/2,2/2}\;\otimes\;\pd{4}{3}{1/3,2/3}\;\otimes\;\pd{5}{4}{2/4},$$
$$B_2 \;\longmapsto\; \pd{1}{1}{1/1}\;\otimes\;\pd{3}{2}{1/2,3/1}\;\otimes\;\pd{4}{2}{3/2}.$$
Their lift product is
$$A_2 * B_2 \;=\; \pd{3}{4}{1/1,1/2,1/3,1/4,2/1,2/2,2/4,3/1,3/2}\,, \qquad \wtt(A_2*B_2)=(4,3,2),$$
the forest RC graph $\{(1,1),(1,2),(1,3),(1,4),(2,1),(2,2),(2,4),(3,1),(3,2)\}$.

\smallskip
Both forest products are forest RC graphs of weight $(4,3,2)$, and one verifies that no other pair $(A,B)$ with the prescribed forest codes produces a forest RC graph of this weight. Hence
$$\beta_{(0,2,3),(2,0,2)}^{(4,3,2)} \;=\; 2,$$
so the coefficient of $\forest_{(4,3,2)}$ in $\forest_{(0,2,3)}\cdot\forest_{(2,0,2)}$ is $2$.

\end{example}

Notice that even in this small example, we have already run into a case where we have two distinct representatives for the corresponding equivalence class. Indeed, while we have achieved a consistent rule in this case by computing products in an associative algebra, it does not descend to a term-by-term associative product for labelings of indexed forests.

\begin{remark}[Dual forest coproduct]
In the companion paper \cite{samuelbialgebra} we study the graded dual $\dcoma=\bigoplus_n\dcoma_n$ of the Schubert bialgebra $\coma=\bigoplus_n\coma_n$, together with the \emph{dual forest basis} $\{\forest_a^*\}$ of $\dcoma_n$ defined by $\langle\forest_a^*,\forest_b\rangle=\delta_{a,b}$. The coproduct $\Delta^*$ of $\dcoma$ is dual to polynomial multiplication, so Theorem~\ref{theorem:forestpolyLR} says equivalently that for a weak composition $c$ of length $n$,
$$\Delta^*(\forest_c^*) \;=\; \sum_{a,b}\beta_{ab}^c\,\forest_a^*\otimes\forest_b^*\qquad\text{in }\dcoma_n\otimes\dcoma_n,$$
exactly as the Schubert structure constants $c_{u,v}^w$ are the structure constants of $\Delta^*$ in the dual Schubert basis \cite[Theorem~\ref*{B-theorem:dschcoproduct}]{samuelbialgebra}. Thus the dual forest coproduct has nonnegative integer structure constants, and Theorem~\ref{theorem:forestpolyLR} gives them a combinatorial meaning. The companion paper also proves an LR rule for the \emph{product} $\sqcupplus$ of $\dcoma$ in the dual forest basis \cite[Theorem~\ref*{B-theorem:LRforest}]{samuelbialgebra}, which corresponds under the Borel isomorphism below to a branching rule $H^\bullet(\mathrm{QFl}_{p+q})\to H^\bullet(\mathrm{QFl}_p)\otimes H^\bullet(\mathrm{QFl}_q)$ for the forest basis; see \cite[\S\ref*{B-subsection:qflbranch}]{samuelbialgebra}.
\end{remark}

\subsection{Geometric interpretation: cup product on the quasisymmetric flag variety}\label{subsection:qflbranch}

The forest polynomials acquire a geometric meaning through the recent work of Bergeron, Gagnon, Nadeau, Spink, and Tewari \cite{bergeron2025qsymflag}, in which the authors introduce the \emph{quasisymmetric flag variety} $\mathrm{QFl}_n\subset \mathrm{Fl}_n$. This is a (reducible) toric complex built as a union
$$\mathrm{QFl}_n \;=\; \bigcup_{T\in\mathrm{Tree}_n} X(T)$$
of left-translated Richardson varieties indexed by $n$-leaf planar binary trees. Their main results give a Borel-type presentation
$$H^\bullet(\mathrm{QFl}_n) \;\cong\; \mathbb{Z}[x_1,\ldots,x_n]/\mathrm{QSym}_n^+,$$
where $\mathrm{QSym}_n^+$ is the ideal of quasisymmetric polynomials with no constant term \cite[Theorem~A and Corollary~12.5]{bergeron2025qsymflag}, and identify the forest polynomials $\{\forest_F : F\in \mathsf{Forest}_n\}$ as a free $\mathbb{Z}$-basis of this ring \cite[Corollary~12.5]{bergeron2025qsymflag}, Kronecker dual to the homology basis $\{[X(F)] : F\in\mathsf{Forest}_n\}$ of $H_\bullet(\mathrm{QFl}_n)$ \cite[Corollary~11.6, Theorem~12.12]{bergeron2025qsymflag}. Under the standard bijection between $\mathsf{Forest}_n$ and weak compositions of length $n$ (compare \cite{nadeau2024forest}), the basis $\forest_F$ matches our composition-indexed basis $\forest_a$.

In this geometric language, the product $\forest_a\,\forest_b$ in $\coma_n$ corresponds to the cup product of cohomology classes in $H^\bullet(\mathrm{QFl}_n)$, and Theorem~\ref{theorem:forestpolyLR} translates immediately into the following statement.

\begin{corollary}[Cup product of forest classes on $\mathrm{QFl}_n$]\label{corollary:cupforest}
Let $[\forest_a], [\forest_b]\in H^\bullet(\mathrm{QFl}_n)$ denote the cohomology classes corresponding to the forest polynomials $\forest_a, \forest_b$ under the Borel isomorphism of \cite[Theorem~A]{bergeron2025qsymflag}. Then
$$[\forest_a]\cup [\forest_b] \;=\; \sum_c \beta_{ab}^c\,[\forest_c]\qquad\text{in } H^\bullet(\mathrm{QFl}_n),$$
where $\beta_{ab}^c\in\mathbb{Z}_{\geq 0}$ is the cardinality
$$
\beta_{ab}^c = \#\Bigl\{(A,B)\,:\,
\begin{aligned}
&\code(\wof{A})=\code_\forest(A)=a,\quad \code(\wof{B})=\code_\forest(B)=b,\\
&A*B=R \text{ a forest RC graph with } \code_\forest(R)=\wtt(R)=c
\end{aligned}
\Bigr\}.
$$
Equivalently, the intersection number $\int_{[X(F_c)]} [\forest_a]\cup [\forest_b]$, where $F_c\in\mathsf{Forest}_n$ corresponds to the composition $c$ under \cite[Corollary~11.6, Theorem~12.12]{bergeron2025qsymflag}, equals $\beta_{ab}^c$.
\end{corollary}

\begin{proof}
The Borel isomorphism $\coma_n=\mathbb{Z}[x_1,\ldots,x_n]\twoheadrightarrow H^\bullet(\mathrm{QFl}_n)$ of \cite[Theorem~A and Corollary~12.5]{bergeron2025qsymflag} is a surjective ring homomorphism that sends the forest-polynomial basis of $\coma_n$ to the cohomology classes $[\forest_a]$, and is in particular a bijection on this basis. Applying it to the identity $\forest_a\forest_b = \sum_c \beta_{ab}^c \forest_c$ of Theorem~\ref{theorem:forestpolyLR} gives the displayed cup product formula. The intersection-number reformulation follows from the Kronecker duality between $\{[\forest_c]\}$ and the homology basis $\{[X(F_c)]\}$ \cite[Theorem~12.12]{bergeron2025qsymflag}.
\end{proof}

\begin{remark}
Corollary~\ref{corollary:cupforest} should be read in the same spirit as Theorem~\ref{theorem:forestpolyLR}: positivity of the cup-product structure constants is implicit in \cite{bergeron2025equivariant,bergeron2025qsymflag}, and a positive combinatorial extraction of them via the straightening algorithm of \cite{bergeron2025equivariant} is already available. The new content is the realization of each coefficient as a single cardinality of an explicit set of pairs of RC graphs. The analogous question for $H^\bullet(\mathrm{Fl}_n)$ in the Schubert basis is the classical (and open) problem of producing such a count for the Schubert structure constants, and the contrast is striking: the forest equivalence $\equiv_\forest$ on RC graphs is rigid enough to permit a single-cardinality formula on $\mathrm{QFl}_n$, whereas the corresponding rigidification on $\mathrm{Fl}_n$ does not seem to be available.
\end{remark}

\section*{Notation and conventions}\label{section:notation}

For the reader's convenience we collect here, in one place, the principal symbols used throughout the article. All notation is introduced in context where it is needed; this list is meant as a navigational aid rather than a substitute for the definitions.

\begingroup
\setlist[description]{
  font=\normalfont,
  leftmargin=2em,
  labelindent=0em,
  labelsep=0.5em,
  itemsep=2pt,
  parsep=0pt,
  topsep=4pt,
}

\subsubsection*{Permutations and polynomials}
\begin{description}
\item[$S_\infty$] The group of permutations of $\mathbb{Z}_{>0}$ fixing all but finitely many elements.
\item[$\ell(w)$] The Coxeter length of $w\in S_\infty$.
\item[$\code(w)$, $\code^*(w)$] The (right) Lehmer code of $w$; the code of $w^{-1}$ (Definition~\ref{definition:codestuff}).
\item[$\maxd(w)$] The position of the last descent of $w$.
\item[$\dom_\lambda$] The dominant permutation with $\code^*(\dom_\lambda)=\lambda$.
\item[$\shup^p$] Shift indices up by $p$ (applied to permutations, RC graphs, or letters).
\item[$\coma_n$] The polynomial ring $\mathbb{Z}[x_1,\ldots,x_n]$; $\Lambda_n\subseteq\coma_n$ the symmetric polynomials.
\item[$\sch_w$, $\sch_w^n$] The Schubert polynomial of $w$; the same regarded in $\coma_n$ for $n\geq\maxd(w)$.
\item[$c_{u,v}^w$] Schubert structure constants: $\sch_u\sch_v=\sum c_{u,v}^w\sch_w$.
\item[$e_{p,k}^n$, $e_\alpha^n$, $\mathcal{E}_n$] The elementary symmetric polynomial of degree $p$ in $x_1,\ldots,x_k$; the $n$-elementary monomial indexed by $\alpha$; the elementary basis of $\coma_n$ (Definition~\ref{definition:schubert}).
\item[$E_w^{\alpha;n}$] The coefficient of $e_\alpha^n$ in $\sch_w^n$.
\item[$\forest_a$] The forest polynomial indexed by the weak composition $a$ (\S\ref{section:forestbackground}).
\item[$\mathfrak{F}_a$, $\mathfrak{F}(\mathbf r)$, $\mathfrak F(R)$] The fundamental slide polynomial indexed by $a$; the slide polynomial of a word $\mathbf r$; $\mathfrak F(\mathrm{word}(R))$ (Definition~\ref{definition:slide}).
\item[$\beta_{ab}^c$] Forest polynomial LR coefficients of Theorem~\ref{theorem:forestpolyLR}: $\forest_a\,\forest_b=\sum_c\beta_{ab}^c\,\forest_c$.
\item[$\mathrm{Comp}_n$] The set of weak compositions of length $n$.
\end{description}

\subsubsection*{RC graphs}
\begin{description}
\item[$\RC_n$, $\RC_n(w)$] The set of bounded RC graphs with $n$ rows; those with $\wof{R}=w$ (Definition~\ref{definition:rcgraph}).
\item[$\RC_n^0$, $\RC(w,n)^0$] The subsets of $\RC_n$, $\RC_n(w)$ consisting of RC graphs whose last row is empty.
\item[$\wof{R}$, $\hr(R)$, $\wtt(R)$] The permutation, height (number of rows), and row weight composition of $R$.
\item[$\mathrm{word}(R)$] The reduced word read from $R$ (rows top to bottom, right to left within rows).
\item[$\overline{R}$] The injectified word $\overline{\mathrm{rev}(\mathrm{word}(R))}$ of $R$ (Definition~\ref{definition:forestequiv}).
\item[$\rtt_R(i,j)$] The right root of the position $(i,j)$ in $R$ (Definition~\ref{definition:roots}).
\item[$\mathrm{bottom}(w)$] The bottom RC graph of $w$, of weight $\code(w)$.
\item[$\BRC$, $\BRC_n$] The free abelian group on bounded RC graphs; its graded component spanned by $\RC_n$.
\item[$\mathcal{S}_w(n)$, $\mathrm{Schub}_n$] $\sum_{R\in\RC_n(w)}R\in\BRC_n$; the span of the $\mathcal{S}_w(n)$ (\S\ref{section:brc}).
\item[$\mathrm{ev}$] The evaluation map $\BRC_n\to\coma_n$, $\mathrm{ev}(R)=x^{\wtt(R)}$.
\item[$\zeromap$] The transition map $\RC_n^0\to\RC_{n-1}$ of Definition~\ref{definition:zeromap}, realizing the Lascoux--Sch\"utzenberger transition formula.
\item[$T_n(w)$] The set of permutations occurring in $\sch_w(x_1,\ldots,x_{n-1},0)$, \eqref{eq:Tn}.
\item[$\trm^p(R)$] ``Truncate'': the RC graph obtained from the last $\hr(R)-p$ rows of $R$ (Definition~\ref{definition:cliptrim}).
\item[$\clp^p(R)$] ``Clip'': the RC graph in $\RC_p$ obtained from the first $p$ rows of $R$ by iterating $\zeromap$ (Definition~\ref{definition:cliptrim}).
\item[$\ltb_p$] Little's bumping operator on reduced words (Definition~\ref{definition:littlebump}).
\end{description}

\subsubsection*{Crystal, forest, and slide structures}
\begin{description}
\item[$e_i,f_i$, $\varepsilon_i$, $\varphi_i$] Crystal raising and lowering operators on RC graphs and their string lengths (Definition~\ref{definition:crystal}).
\item[$\mathbb{Y}(R)$, $\mathrm{Dem}(R)$] The highest-weight representative of $R$; the Demazure crystal (connected component) of $R$.
\item[$\ddot e_i,\ddot f_i$, $\ddot\varepsilon_i,\ddot\varphi_i$, $\ddot\otimes$] Quasi-crystal operators on $\RC_n(w)$ and the quasi-tensor product (Definition~\ref{definition:quasicrystal}).
\item[$\finv$, $P(W)$, $Q(W)$, $F(W)$] The forest invariant (LBS labeling), recording labeling, and indexed forest of an injective word $W$ under $\Omega$-insertion (Definition~\ref{definition:forestinvariant}).
\item[$\fcode(R)\,(=\code_\forest(R))$] The forest code of $R$, $\code(F(\overline{R}))$ (Definition~\ref{definition:forestequiv}).
\item[$\equiv_\forest$] Forest equivalence: $R\equiv_\forest R'$ iff $P(\overline{R})=P(\overline{R'})$.
\item[$\mathcal{C}_\forest(w)$, $\mathcal{C}_\forest(a)$] The forest classes of reduced words of $w$; the forest classes of RC graphs in $\RC_n$ with forest code $a$.
\item[$\mathsf{Sh}(W_1,W_2)$] The multiset of shuffles of two injective words.
\item[$\QY(R)$] The quasi-Yamanouchi representative of the slide class of $R$ (Definition~\ref{definition:slide}).
\end{description}

\subsubsection*{Products on RC graphs}
\begin{description}
\item[$R_1\sqcup R_2$] The horizontal disjoint union of two RC graphs of common height (\S\ref{section:multiplactic}).
\item[$\boxtimes$] The squash product $R_1\boxtimes R_2=\clp^{\hr(R_1)}(R_1\sqcup R_2)$ (\S\ref{section:multiplactic}).
\item[$\grass_n$] The span in $\BRC_n$ of the full $n$-Grassmannian RC graphs; a ring under $\boxtimes$ isomorphic to the plactic algebra $\plac_n$ (Theorem~\ref{theorem:placticiso}).
\item[$P_n$] The isomorphism $\mathbb{Z}^{\oplus\SSYT_n}\to\grass_n$ (\S\ref{section:multiplactic}).
\item[$\mathcal{M}_n$, $\RC_n(-)$] The multiplactic algebra $\bigotimes_{k=1}^n\grass_k$ and the projection $\mathcal{M}_n\to\BRC_n$ (\S\ref{section:multiplactic}).
\item[$\mathbf{E}^{k;\beta}$, $\mathbf{E}_p^k$] The elementary symmetric full $k$-Grassmannian RC graph with rows indexed by the support of $\beta\in\{0,1\}^k$; the sum of those of degree $p$.
\item[$\mathbf{S}_w(n)$] The signed lift of $\mathcal{S}_w(n)$ to $\mathcal{M}_n$ (\S\ref{section:multiplactic}).
\item[$\mathrm{lift}_n$] The canonical lift $\RC_n\to\mathcal{M}_n$ (\S\ref{section:liftproduct}).
\item[$*$] The lift product $R_1*R_2=\RC_n(\mathrm{lift}_n(R_1)\mathrm{lift}_n(R_2))$ on $\BRC_n$ (\S\ref{section:liftproduct}).
\item[$\sqcupplus$] The concatenation product on $\BRC$ of the companion paper \cite{samuelbialgebra}; not used in this paper except in \S\ref{section:cliptrim}.
\end{description}

\endgroup

\section*{Acknowledgements}

The author is deeply grateful to Zachary Hamaker for many illuminating conversations during the development of this work, and in particular for introducing the author to forest polynomials in the first place, in addition to pointing out the connections of RC graph crystal operators to  Little's bumping algorithm. Thanks are also due to Frank Sottile, whose work, much of it now classical, has been a sustained guiding influence on the author's thinking about Schubert calculus, and whose generosity with his time over the years has shaped the perspective taken here.

\bibliographystyle{acm}
\bibliography{schubert_bialgebra,forest_lr_rule_separate_paper}

\end{document}